\documentclass[11pt]{amsart}

\usepackage{amssymb,latexsym,amsmath,amsthm,amscd,verbatim,mathrsfs}

\usepackage[left=3cm,top=3cm,right=3cm,bottom = 3cm]{geometry}

\usepackage{tikz-cd}
\tikzset{>=latex}

\usepackage{hyperref}
\usepackage{xcolor}
\hypersetup{
    colorlinks,
    linkcolor={red!50!black},
    citecolor={blue!50!black},
    urlcolor={blue!80!black}
}

\newtheoremstyle{def}
     {10pt}
     {10pt}
     {}
     {}
     {\rmfamily\bfseries\upshape}
     {.}
     {.5em}
     {}
 \theoremstyle{def}
 \newtheorem{definition}{Definition}[subsection]
\newtheorem{remark}[definition]{Remark}
\newtheorem{example}[definition]{Example}

\newtheoremstyle{thm}
     {20pt}
     {10pt}
     {\it}
     {}
     {\rmfamily\bfseries\upshape}
     {.}
     {.5em}
     {}

\theoremstyle{thm}

\newtheorem{theorem}[definition]{Theorem}
\newtheorem{lemma}[definition]{Lemma}
\newtheorem{proposition}[definition]{Proposition}
\newtheorem{corollary}[definition]{Corollary}
\newtheorem{principle}[definition]{Principle}

\DeclareMathOperator{\Aut}{Aut}

\DeclareMathOperator{\im}{im}

\DeclareMathOperator{\Spec}{Spec}

\DeclareMathOperator{\Hom}{Hom}

\DeclareMathOperator{\id}{id}

\DeclareMathOperator{\pr}{pr}
\DeclareMathOperator{\Pic}{Pic}
\DeclareMathOperator{\SL}{SL}
\DeclareMathOperator{\Mp}{Mp}
\DeclareMathOperator{\Sp}{Sp}
\DeclareMathOperator{\GL}{GL}

\newcommand{\smalltwobytwo}[4]{
\bigl( \begin{smallmatrix} 
  #1 & #2\\
  #3 & #4 
\end{smallmatrix} \bigr)}

\newcommand{\ppair}[2]{
\left\langle #1,#2  \right\rangle}

\newcommand{\un}[1]{\underline{#1}}

\newcommand{\cF}{\mathcal{F}}

\newcommand{\cJ}{\mathcal{J}}
\newcommand{\cK}{\mathcal{K}}
\newcommand{\cL}{\mathcal{L}}
\newcommand{\cM}{\mathcal{M}}

\newcommand{\cO}{\mathcal{O}}

\newcommand{\cV}{\mathcal{V}}
\newcommand{\cW}{\mathcal{W}}

\newcommand{\sA}{\mathscr{A}}
\newcommand{\sE}{\mathscr{E}}

\newcommand{\sM}{\mathscr{M}}

\newcommand{\sX}{\mathscr{X}}

\newcommand{\CC}{\mathbb{C}}

\newcommand{\GG}{\mathbb{G}}

\newcommand{\ZZ}{\mathbb{Z}}

\usepackage{pgf, tikz}
\usetikzlibrary{arrows, automata}

\makeatletter
\def\@seccntformat#1{%
  \protect\textup{\protect\@secnumfont
    \ifnum\pdfstrcmp{subsection}{#1}=0 \bfseries\fi
    \csname the#1\endcsname
    \protect\@secnumpunct
  }%
}  
\makeatother

\numberwithin{equation}{section}

\DeclareMathOperator{\mf}{\underline{\omega}}

\DeclareMathOperator{\Des}{Des}

\usepackage{tikz-cd}
\tikzset{>=latex}

\begin{document}

\title[Transformation laws]{The transformation laws of algebraic theta functions}

\author{Luca Candelori}
\email{lcandelori@lsu.edu}

\begin{abstract}
We give a comprehensive treatment of the transformation laws of theta functions from an algebro-geometric perspective, that is, in terms of moduli of abelian schemes. This is accomplished by introducing geometric notions of theta-descent structures, metaplectic stacks, and bundles of half-forms, whose analytic incarnations underlie different aspects of the  classical transformation laws. As an application, we lay the foundations for a geometric theory of modular forms of half-integral weight and, more generally, for modular forms taking values in the Weil representation. We discuss further applications to the algebraic theory of Jacobi forms and to the theory of determinant bundles of abelian schemes.   
\end{abstract}
\maketitle

\section{Introduction}

Classical theta functions are viewed geometrically as coordinates of canonical embeddings of a $g$-dimensional complex torus  into projective space, giving the torus the structure of an abelian variety.  The set of complex structures on the torus can be parametrized by a matrix $\tau$ in Siegel's upper half-space $\mathfrak{h}_g$. As $\Sp_{2g}(\ZZ)$ acts on $\mathfrak{h}_g$, the theta functions undergo certain transformations with respect to $\tau$: in particular, theta functions are {\em modular forms of half-integral weight} on the {\em metaplectic group $\Mp_{2g}(\ZZ)$}, the unique non-trivial double cover of $\Sp_{2g}(\ZZ)$. These `transformation laws' can be computed analytically, and their formulas are well-known. However, (paraphrasing Andr\'{e} Weil \cite{Weil}) the role of the metaplectic group as a `deus-ex-machina' underlying the transformation laws is still shrouded in mystery, particularly from a geometric perspective. The purpose of this article is to expose the mechanisms of the transformation laws by working algebraically, thus replacing the classical theta functions with their algebraic avatars. In fact, we forget completely about the functions themselves, and instead interpret the transformation laws as canonical isomorphisms between vector bundles over moduli stacks of abelian schemes. 

Our work is based on the theory of algebraic theta functions, as first laid out by David Mumford in the series of papers ``{\em On the equations defining abelian varieties}" \cite{Mumford:EqAbv1}, \cite{Mumford:EqAbv2}, \cite{Mumford:EqAbv3}. Indeed, it was clear to Mumford that his theory could be employed to `explain' the transformation laws: 

\begin{quote}``There are several interesting topics which I have not gone into in this paper, but which can be investigated in the same spirit: for example, [...] a discussion of the transformation theory of theta-functions..." (D. Mumford, \cite{Mumford:EqAbv1}, p.287).
\end{quote}

The present work can thus be viewed as a natural addendum to \cite{Mumford:EqAbv1}, \cite{Mumford:EqAbv2}, \cite{Mumford:EqAbv3}, and we have made an extensive effort to keep the terminology and notation as close as possible to those original papers. We do however make extensive use of the language of algebraic stacks, whose theory has been thoroughly developed in the intervening years. 

We begin in Section \ref{section:Theta-Descent} by reviewing Mumford's theory of theta structures (\cite{Mumford:EqAbv1},\cite{Mumford:EqAbv2}), and introduce a variation of it which we call `theta-descent structures'. These types of level structures on abelian schemes are strictly weaker than theta structures, in the same sense as, for example, $\Gamma_0(N)$-structures on elliptic curves are weaker than full level $N$ structures. Their purpose is to control the descent of line bundles on abelian schemes under isogenies, as opposed to `full' theta structures, which control the representation theory of theta groups. 

Out of the theory of theta-descent structures there emerges the theory of modular forms of half-integral weight, as  explained in Section \ref{section:metaplecticStacks}. In particular, we combine theta-descent with our earlier work on theta multiplier bundles (\cite{C1}) to obtain an algebraic theory of modular forms of half-integral weight, in the sense of Shimura \cite{Shimura}, and natural generalizations to higher type and degree (Def. \ref{definition:GeneralModForms}). This application was originally motivated by the work of Nick Ramsey (\cite{Ramsey}), who has studied algebraic and $p$-adic analogs of modular forms of half-integral weight in the case $g=1$ and type $\delta = (2)$ (the `Shimura' modular forms of \cite{Shimura}, Def. \ref{definition:ShimuraModularForm}). Ramsey's modular forms are defined as sections of line bundles corresponding to the divisors of theta functions. Our approach is different in that it is more general and it is purely modular, i.e. the relevant line bundles can be constructed directly out of the moduli problem of theta-descent, without requiring theta functions or their divisors. On the other hand, Ramsey's theory includes a discussion of compactifications and Hecke operators (as well as beautiful $p$-adic applications) which our theory currently lacks. In any case, our results suggest the existence of yet-to-be defined algebraic structures which could lead to progress in studying the Fourier coefficients of modular forms of half-integral weight. Section \ref{section:metaplecticStacks} also contains the construction of the main technical tools that we employ, namely the theory of metaplectic stacks and a review of theta multiplier bundles (\cite{C1}), leading up to the definition of the bundle of half-forms (Def. \ref{definition:half-forms}). All these constructions have been inspired by their analogs in geometric quantization theory, though for brevity we have omitted any computation linking the two explicitly.   

Section \ref{section:WeilBundles} contains perhaps the most important result of this article (Thm. \ref{theorem:IdealTheorem}), giving a canonical, algebraic isomorphism $\iota_{\rm{alg}}$ between two vector bundles over the moduli stack of abelian schemes of dimension $g$ with a polarization of even type $\delta$. On one side of the isomorphism we have the {\em Weil bundle} $\cW_{g,\delta}$, whose analytic incarnation is the local system defined by the Weil representation of \cite{Weil}; on the other side we have the bundle $\cJ^{\vee}_{g,\delta}$, whose analytic sections are the vectors of theta constants (theta-nulls). We prove that over the analytic category the algebraic isomorphism $\iota_{\rm{alg}}$ of Thm. \ref{theorem:IdealTheorem}, tensored with $\CC$, essentially coincides with the isomorphism $\iota_{\rm{an}}$ given by the transformation laws of theta functions, and thus can be viewed as a rather precise algebraic analog for the transformation laws. The sections of the Weil bundle are also interesting on their own, and provide new, algebraic analogs of {\em vector-valued} modular forms taking values in the Weil representation (Def. \ref{definition:vvaluedModularForms}), in the sense of Borcherds (\cite{Borcherds}) and others. At the end of Section \ref{section:WeilBundles}, we also briefly discuss possible applications of Thm. \ref{theorem:IdealTheorem} to the theory of Jacobi forms and the theory of determinant bundles of abelian schemes.   

At the end of every section we have included an `analytic theory' sub-section where we compute our algebraic constructions explicitly over the category of analytic spaces. In some cases, we recover well-known constructions associated to the classical transformation laws of theta functions (e.g. metaplectic stacks correspond to quotients of Siegel's upper half-space by metaplectic groups, theta multiplier bundles correspond to the characters of theta groups appearing in the transformation laws, half-forms correspond to modular forms of half-integral weight, etc.). In others, such as for the theory of theta-descent structures, our analytic constructions seem new. For simplicity, these analytic computations all take place in the case $g=1$ (i.e. the case of elliptic curves), but entirely analogous computations can be carried out in higher degree.  

Throughout this work, we have restricted our attention to totally symmetric line bundles over abelian varieties, in the same spirit as \cite{Mumford:EqAbv1}, \cite{Mumford:EqAbv2}, \cite{Mumford:EqAbv3}. Extending the theory to the more general case of symmetric line bundles should not present any difficulty. A more tantalizing topic to investigate would be the extension to the case of {\em non-degenerate} line bundles, i.e. those line bundles whose theta group is an extension of a finite flat group scheme. In this case, Thm. \ref{theorem:IdealTheorem} should still hold with the right-hand side replaced by a higher derived direct image. It would then be interesting to investigate the algebraic nature of these `generalized' transformation laws and the corresponding theory of modular forms of half-integral weight, whose definition can no longer be given in terms of classical theta functions.

\subsection{Notation}
Throughout, we work over the category of schemes $S$ which are separated and locally noetherian, and we endow this category with the (big) \'{e}tale topology. If $S$ is a scheme and $G$ is a group, we denote by $\un{G}$ the constant \'{e}tale sheaf obtained by sheafification of the functor on $S$-schemes given by $T\mapsto G$. In terms of notation, we denote schemes by capital letters (e.g. $A,S,T$...), algebraic stacks by scripts (e.g. $\sA, \sX, \sM$...) and sheaves of modules by calligraphic letter (e.g. $\cL, \cV, \cF$...). For a locally free $\cO_S$-module $\cF$, we denote its dual in the category of locally free sheaves by $\cF^{\vee} := Hom_{\cO_S}(\cF,\cO_S)$, and a similar notation will be employed when the base scheme is replaced by an algebraic stack.

\subsection{Acknowledgments}

This work was initiated while I was a graduate student of Henri Darmon at McGill University, thus I would like to thank him for his support. At McGill, I would also like to thank Eyal Goren and Niky Kamran for their interest and helpful insights to earlier versions of this work. I would also like to thank William J. Hoffman, Ben Howard, Ling Long and Nick Ramsey for many clarifying conversations about the topic.   

\section{Theta- and Theta-Descent structures}
\label{section:Theta-Descent}

\subsection{} Let $\pi:A\rightarrow S$ be an abelian scheme of dimension $g$ with identity section $e:S\rightarrow A$. Let $\cL$ be a relatively ample, invertible $\cO_A$-module which has been normalized at the identity, i.e. it is equipped with a fixed $\cO_S$-module isomorphism
$
e^*\cL \simeq \cO_S  
$
(in the following, any isomorphism of normalized invertible sheaves is assumed to respect the normalizations, so that, in particular, a normalized invertible sheaf has no automorphisms). For $P$ a scheme-valued point of $A$, let $T_P$ be the translation-by-$P$ morphism and consider the canonical morphism
$$
\phi_{\cL}: A \longrightarrow A^{t} , \quad 
 P \longmapsto T_P^*\cL\otimes\cL^{-1} 
$$
to the dual abelian scheme induced by $\cL$. The group scheme 
$$
K(\cL):= \ker \phi_{\cL}
$$
is commutative, finite flat over $S$ and it is canonically endowed with a non-degenerate symplectic pairing
$$
e_{\cL}: K(\cL)\times K(\cL) \longrightarrow \GG_m. 
$$
In this case $\mathrm{rk}(K(\cL)) = d^2$, for some integer $d>0$ called the {\em degree} of $\cL$. $K(\cL)$ is then  finite \'{e}tale over $S[1/d]$ and  \'{e}tale locally we may find an isomorphism 
$$
K(\cL) \simeq \left(\prod_{i=1}^g \underline{\ZZ/d_i\ZZ}\right)\times \left(\prod_{i=1}^g \mu_{d_i}\right)
$$ 
for some sequence of integers $\delta := (d_1,\ldots,d_g)$ with $d_1\mid d_2 \mid \cdots \mid d_g$ and $d = d_1\cdots d_g$
(\cite{Mumford:EqAbv1} \S 1, \cite{Mumford:EqAbv2} \S 6). The sequence $\delta$ is called the {\em type} of $\cL$. 

For any scheme $T\rightarrow S$, let $\Aut_{A\times_S T}(\cL/A)$ be the group of automorphisms of $\cL$ covering a translation morphism of $A\times_S T$. The functor $(T\rightarrow S)\mapsto \Aut_{A\times_S T}(\cL/A)$ is represented by group scheme
$
G(\cL),
$
the {\em theta group} of $\cL$, which is a central extension 
\begin{equation}
\label{eqn:thetaGroupExactSequence}
1 \rightarrow \GG_m \rightarrow G(\cL) \stackrel{p}\rightarrow K(\cL) \rightarrow 1,
\end{equation}
with commutator pairing equal to $e_{\cL}$ (\cite{Mumford:EqAbv2} \S 6, \cite{AbelianV} Thm. 23.1). The scheme-valued points of $G(\cL)$ can be represented by pairs $(z,\varphi)$, where $\varphi: T^*_z\cL \stackrel{\simeq}\rightarrow \cL$.

 The key feature of theta groups is that they control the descent of invertible sheaves under isogenies of abelian schemes: we briefly recall this result, which will be refined by Thm. \ref{theorem:DescentForSymmetricLineBundles} below. 

\begin{theorem}[\cite{AbelianV}, Thm. 23.2, \cite{MoretBailly:PinceauxAbv} Prop. VI.1.2]
\label{theorem:DescentForLineBundles}
Let $A\rightarrow S[1/d]$ be an abelian scheme and let $\cL$ be a relatively ample, normalized invertible sheaf of degree $d$. Let $\phi:A\rightarrow B$ be an isogeny of abelian schemes, and let $H:= \ker f$. Then there is an equivalence between: 
\begin{itemize}
\item[(a)] The category of pairs $(\cM, f)$ of  relatively ample, normalized invertible sheaves $\cM$ over $B$ such that $f: \phi^*\cM \stackrel{\simeq}\rightarrow \cL$.
\item[(b)] The discrete category of group homomorphisms $\sigma: H \rightarrow G(\cL)$ making the following diagram 
$$
\begin{tikzcd}
  &H \arrow{dl}{\sigma} \arrow[hook]{dr} &   \\
G(\cL) \arrow{r}{p} & K(\cL) \arrow[hook]{r}& A
\end{tikzcd}
$$
commutative. 
\end{itemize}
\end{theorem}

\begin{proof}
Suppose first that a pair $(\cM,f)$ is given with $f: \phi^*\cM \stackrel{\simeq}\rightarrow \cL$. For any $h\in H\subseteq A$, we have 
$$
T_h^*\cL \stackrel{T_h^*f^{-1}}\rightarrow T_h^*\phi^*\cM = (\phi\circ T_h)^*\cM = \phi^*\cM \stackrel{f}\rightarrow \cL,
$$
thus $(h,f\circ T_h^*f^{-1}) \in G(\cL)$
is an object of category (b), and we have defined a functor from category (a) to category (b). Since both categories are rigid groupoids, it suffices to show that this functor is a bijection at the level of equivalence classes of objects. Therefore, suppose that a an object $(H,\sigma)$ of category (b) is given, so that we have an action of $\sigma H$ on $\cL$ covering the translation action of $H$ on $A$. The sheaf $\phi_*\cL$, which is locally free of rank $\mathrm{rk}\, H$, is then endowed with an action of $\sigma H$. Define $\cM:= (\phi_*\cL)^{\sigma H}$, the invariants under this action, and let $f$ be the canonical isomorphism $f: \phi^*(\phi_*\cL)^{\sigma H}\stackrel{\simeq}\rightarrow \cL$ (\cite{AbelianV}, Prop. II.7.2). 
\end{proof}

\begin{definition}
An object $(H,\sigma)$ as in Thm. \ref{theorem:DescentForLineBundles} (b) is called a {\em splitting} of the theta group $G(\cL)$ over $H$. 
\end{definition}

Let $H:=\ker \phi$, with $\phi:A\rightarrow B=A/H$ as in Thm. \ref{theorem:DescentForLineBundles}, and suppose a splitting $\sigma$ is given. A natural question to ask is to express the theta group $G(\cM)$ of the descended invertible sheaf over $B$ in terms of the pair $(H,\sigma)$.  Note first that $H \subseteq K(\cL)$ must necessarily be isotropic with respect to the pairing $e_{\cL}$. By \cite{AbelianV}, Lemma 23.2, we then have 
$$
K(\cM)= H^{\perp}/H, \quad G(\cM) =  p^{-1}(H^{\perp})/\sigma(H).
$$
In particular if $H^{\perp} = H$ (i.e. $H$ is {\em maximal} isotropic) then $\deg \cM = 1$, so that $\cM$ induces a principal polarization $\phi_{\cM}:B\rightarrow B^t$.  

\begin{definition}
A splitting $(H,\sigma)$ of $G(\cL)$ over $H$ is {\em maximal}  whenever $H^{\perp} = H$.
\end{definition}

\subsection{}
\label{subsection:symmetricLineBundles}
Suppose next that $S$ is a scheme where $1/2\in \cO_S$ and suppose that the (relatively ample, normalized) invertible sheaf $\cL$ is {\em symmetric}. This means that there is a (unique) isomorphism of normalized invertible $\cO_A$-modules
$
\iota_{\cL}: [-1]^*\cL \stackrel{\simeq}\longrightarrow \cL,
$
where $[-1]:A \rightarrow A$ is the inversion morphism. The map $\iota_{\cL}$ then induces an order two automorphism 
$$
\delta_{-1}: G(\cL) \stackrel{\iota^{-1}_{\cL}}\longrightarrow [-1]^*G(\cL)\longrightarrow G(\cL),
$$
restricting to the identity on $\GG_m$. In terms of points $(z,\varphi) \in G(\cL)$, we have (\cite{Mumford:EqAbv1}, \S 2)
\begin{equation}
\label{equation:delta-1}
\delta_{-1}(z,\varphi) = (-z,(T^*_{-x}\iota_{\cL})\circ [-1]^*\varphi\circ \iota_{\cL}^{-1}).
\end{equation}
We now refine Thm. \ref{theorem:DescentForLineBundles} by taking into account the additional automorphism $\delta_{-1}$, which controls the descent of symmetric invertible sheaves. 

\begin{theorem}
\label{theorem:DescentForSymmetricLineBundles}
Let $(A\rightarrow S[1/2d],\cL)$ be as in Thm. \ref{theorem:DescentForLineBundles} and suppose additionally that $\cL$ is symmetric. Let $f:A\rightarrow B$ be an isogeny of abelian schemes, and let $H:= \ker f$.  Then there is an equivalence between:
\begin{itemize}
\item[(a)] The category of pairs $(\cM, f)$ of symmetric, relatively ample, normalized invertible sheaves $\cM$ over $B$ such that $f: \phi^*\cM \stackrel{\simeq}\rightarrow \cL$.
\item[(b)] The category of splittings $(H,\sigma)$ of $G(\cL)$ over $H$ such that the following diagram 
$$
\begin{tikzcd}
G(\cL) \arrow{rr}{\delta_{-1}} 
& & G(\cL)  \\
&H \arrow{ur}{\sigma\circ[-1]} \arrow{ul}{\sigma} &
\end{tikzcd}
$$
is commutative.
\end{itemize}
\end{theorem}

\begin{proof}
As in Thm.  \ref{theorem:DescentForLineBundles}, it suffices to show a bijection at the level of objects. First, given $(\cM,f)$ note that the splitting $\sigma: H \rightarrow G(\cL)$ given by $(h, T^*_h f^{-1}\circ f)$ lies in category (b), as follows from by explicit computation using \eqref{equation:delta-1}. Conversely, suppose that a splitting $(H,\sigma)$ in (b) is given, and let $\cM:= \phi_*\cL^{\sigma H}$, $f: \phi^*(\phi_*\cL)^{\sigma H}\stackrel{\simeq}\rightarrow \cL$ as in Thm. \ref{theorem:DescentForLineBundles}. We first need:

\begin{lemma}
\label{lemma:symmetricDescentLemma}
There is a canonical isomorphism $\phi_*[-1]_A^*\cL \simeq [-1]_B^*\phi_*\cL$. 
\end{lemma}

\begin{proof}
Note that $\phi\circ [-1]_A = [-1]_B\circ\phi$, since the morphism $\phi$ must respect the structure of abelian schemes. By functoriality there is a corresponding canonical isomorphism of functors  
$
\phi_*[-1]_{A*} \simeq [-1]_{B*}\phi_*.
$
Applying this isomorphism to $[-1]_A^*\cL$ we get
$$
\phi_*[-1]_{A*}[-1]^*_A\cL  = \phi_*\cL\simeq [-1]_{B*}\phi_*[-1]_A^*\cL,
$$
using the fact that the isogeny $[-1]_A$ is of degree 1. Finally, we may pull back both sides by $[-1]_B^*$ and again use the fact that $ [-1]_B^*[-1]_{B*} = 1$ to get the desired result. 
\end{proof}

{\em End of the proof of Thm. \ref{theorem:DescentForSymmetricLineBundles}.} Applying invariants to the isomorphism of Lemma \ref{lemma:symmetricDescentLemma}, we get a canonical isomorphism 
$$
\phi_*([-1]_A^*\cL)^{\sigma[-1]_A H} = [-1]^*_B\phi_*(\cL)^{\sigma H}.
$$
But by assumption, the action of $\sigma[-1]_A H$ on $\phi_*[-1]^*_A\cL$ is equivariant with the action of $\sigma H$ on $\phi_*\cL$ under the isomorphism $\iota_{\cL}: [-1]^*\cL \stackrel{\simeq}\rightarrow \cL$, so that $\phi_*([-1]_A^*\cL)^{\sigma[-1]_A H}\simeq \phi_*\cL^{\sigma H} = \cM$. Thus we have found the required isomorphism $[-1]_B^*\cM \simeq \cM$.  
\end{proof}

\begin{definition}
\label{definition:symmetricSplitting}
A {\em symmetric splitting} $(H,\sigma)$ is a splitting of $G(\cL)$ over $H$ as in Thm. \ref{theorem:DescentForSymmetricLineBundles} (b).  A {\em maximal} symmetric splitting is a symmetric splitting with $H^{\perp} = H$. 
\end{definition} 

\begin{remark}If $(H,\sigma)$ is a maximal symmetric splitting, the descended invertible sheaf $\cM$ induces a symmetric principal polarization $\phi_{M}:B\rightarrow B^t$. Following standard notation, in this case we denote $\cM$ by $\Theta$. Indeed if $S = \Spec(k)$ with $\bar{k} = k$ and $\mathrm{char}(k) = 0$, the invertible sheaf $\Theta$ corresponds to a classical {\em theta divisor} on the abelian variety $B/k$.   
\end{remark}

The isomorphism $\iota_{\cL}$, restricted to the fixed locus $A[2]$ of $[-1]^*$, is multiplication by $\pm 1$ and thus it defines a function
$$
e^{\cL}_*: A[2] \longrightarrow \mu_2,
$$
the {\em theta characteristic} of $\cL$. This function is quadratic for the symmetric pairing $e_{\cL^2}$ (\cite{Mumford:EqAbv1},\S 2, Cor. 1) and, along with $e_{\cL}$, they completely determine the isomorphism class of $\cL$. Given a symmetric splitting $(H,\sigma)$, the theta characteristic $e_*^{\cM}$ of the descended symmetric invertible sheaf $\cM$ over $B$ can be computed explicitly in terms of $(H,\sigma)$. To do so, consider the homomorphisms (\cite{Mumford:EqAbv1}, \S 2)
\begin{equation}
\begin{aligned}
\label{equation:eta2}
\epsilon_2: G(\cL) \rightarrow G(\cL^{\otimes 2}) , \quad \eta_2: & G(\cL^{\otimes 2}) \rightarrow G(\cL) \\
(z,\varphi) \mapsto (z, \varphi^{\otimes 2})  \quad\quad & (z,\varphi)\mapsto (z,\rho),
\end{aligned}
\end{equation}
where $\rho: T_{2z}^*\cL \stackrel{\simeq}\rightarrow\cL$ is the unique isomorphism such that $[2]^*\rho = \nu\circ \varphi^{\otimes 2} \circ T_z^*\nu^{-1}$ and $\nu: \cL^{\otimes 4} \stackrel{\simeq}\rightarrow [2]^*\cL$ is the isomorphism induced by symmetry (\cite{AbelianV}, Cor. 6.3). Then at the level of points the following diagram of sets 
$$
\begin{tikzcd}
A[2] \arrow{rr}{s}\arrow{dr}{e_*^{\cL}} & & G(\cL^{\otimes 2})[2]\arrow{dl}{\eta_2}  \\
&\mu_2  &
\end{tikzcd}
$$
is commutative, where $s$ is any (set-theoretic) lift of the 2-torsion of $A$ to the 2-torsion of $G(\cL^{\otimes 2})$ (\cite{Mumford:EqAbv1}, Prop. 2.6). Consequently, we may compute $e_*^{\cM}$ on points via the commutative diagram  
$$
\begin{tikzcd}
A/H[2]=B[2] \arrow{rr}{s}\arrow{dr}{e_*^{\cM}} & & G(\cM^{\otimes 2})[2] = p^{-1}_{G(\cL^{\otimes 2})}(\epsilon_2H^{\perp})/\epsilon_2\sigma (H)[2] \arrow{dl}{\eta_2}  \\
&\mu_2.  &
\end{tikzcd}
$$

\subsection{} Given a (relatively ample, normalized) symmetric invertible sheaf $\cL$ over $A$, it is not necessarily the case that a symmetric splitting $(H,\sigma)$ of $G(\cL)$ over $H\subseteq K(\cL)$ exists. To ensure existence, we must impose on $A$ a type of level structure that we call a {\em $\Theta$-descent structure}, a construction based on the theory of theta structures (\cite{Mumford:EqAbv1} \S 1, \cite{Mumford:EqAbv2} \S 6), which we briefly recall. Let $\delta := (d_1,\ldots,d_g)$ be a type with $d = d_1\cdots d_g$ and let
$$
H(\delta):= \prod_{i=1}^g \ZZ/d_i\ZZ,\quad 
K(\delta):= H(\delta)\times H(\delta).
$$
The elements of $K(\delta)$ are denoted by $z = (x,y)$, with $x,y \in H(\delta)$. Let $\{e_1, \ldots, e_{2g}\}$ be the standard basis of $K(\delta)$. Let $\zeta_d \in \mu_d$ be a primitive $d$-th root of unity and define the {\em standard symplectic pairing of type $\delta$} to be the non-degenerate alternating bilinear pairing on $K(\delta)$ determined by
\begin{equation}
\label{equation:standardSymplecticPairing}
\ppair{e_{\nu}}{e_{\mu}}_{\delta} = \left\{ \begin{array}{lc} 
\zeta_d^{-1} &\text{ if } \mu = g + \nu \\
\zeta_d &\text{ if } \nu = g + \mu \\
1 &\text{ otherwise. }  \end{array}   \right.
\end{equation}
If $S$ is a scheme containing $1/d$ and a primitive $d$-th root of unity $\zeta_d \in H^0(S,\cO_S^{\times})$, we may form the constant group scheme $\un{K}(\delta)$ together with the standard symplectic pairing, viewed as a morphism of constant group schemes $\ppair{}{}_{\delta}:\un{K}(\delta)^{\times 2}\rightarrow \mu_d$.  

\begin{definition}
\label{definition:HeisenbergGroupScheme}
The {\em Heisenberg group scheme of type} $\delta = (d_1,\ldots,d_g)$ over $S$ is the flat affine scheme 
$$
G(\delta):= \GG_m\times_S \underline{K}(\delta),
$$ 
with group structure given by 
$
(\lambda_1, x_1,y_1)(\lambda_2, x_2,y_2) = (\lambda_1\lambda_2\,\langle x_1,y_2\rangle_{\delta}  , x_1+x_2,y_1 + y_2).
$
\end{definition}

Note that $G(\delta)$ is a central extension of $\underline{K}(\delta)$ by $\GG_m$ of commutator pairing equal to $\langle,\rangle_{\delta}$. If $\cL$ is of type $\delta$, then (\'{e}tale) locally on $S$ the theta group $G(\cL)$ is  isomorphic as a central extension to $G(\delta)$ (\cite{Mumford:EqAbv1}, \S 1), which motivates the following definition:

\begin{definition}
Let $A\rightarrow S$ be an abelian scheme of dimension $g$ over a scheme $S$ containing $1/d$ and $\zeta_d$, and let $\cL$ be a relatively ample, normalized invertible sheaf over $A$ of type $\delta$. A {\em theta structure} on $(A,\cL)$ is a group scheme isomorphism $$
\beta: G(\cL) \stackrel{\simeq}\rightarrow G(\delta),
$$
fitting in the following commutative diagram 
$$
\begin{tikzcd}
1 \arrow{r} & \GG_m \arrow{r} \arrow{d}{\mathrm{id}} &G(\cL) \arrow{r} \arrow{d}{\beta} & K(\mathcal{L}) \arrow{r} \arrow{d}{\alpha} & 1 \\
1 \arrow{r} & \GG_m \arrow{r} &G(\delta) \arrow{r} & \underline{K}(\delta) \arrow{r} & 1. 
\end{tikzcd}
$$
\end{definition}

\begin{remark}
By a straightforward computation, the induced isomorphism $$\alpha: K(\cL) \stackrel{\simeq}\rightarrow \underline{K}(\delta)$$ is an {\em arithmetic level structure}, i.e. a group scheme isomorphism such that $\alpha^*\ppair{}{}_{\delta} = e_{\cL}$. 
\end{remark}

Consider the \'{e}tale sheaf of groups defined by the functor on $S$-schemes 
$$
\Aut_{\GG_m}G(\delta)(T\rightarrow S):= \{ f:G(\delta)(T)\stackrel{\simeq}\rightarrow G(\delta)(T) \text{ such that } f|_{\GG_m(T)} = \id \}.
$$
By the same arguments as in \cite{BL}, Lemma 6.6.6, we have:

\begin{lemma}
\label{lemma:AutExactSequence}
There is an exact sequence of \'{e}tale sheaves of groups
$$
0 \rightarrow \underline{K}(\delta) \stackrel{i}\rightarrow \Aut_{\GG_m}G(\delta) \stackrel{p}\rightarrow \Sp(\un{K}(\delta),\ppair{}{}_{\delta})\rightarrow 0,
$$
where $i$ is given on points by $i(z) =\{ (\lambda,z') \mapsto (\lambda\,\ppair{z}{z'}_{\delta},z')\}$ and the second map is induced by the projection $p: G(\delta)\rightarrow \underline{K}(\delta)$.
\end{lemma}

The \'{e}tale sheaf of groups $\Aut_{\GG_m}G(\delta)$ is thus representable by a constant group scheme over $S$, which implies:

\begin{corollary}
\label{corollary:TorsorOfThetaStructuresIsFiniteEtale}
The $\Aut_{\GG_m}G(\delta)$-torsor
$$
\mathrm{Isom}_{\GG_m}(G(\cL),G(\delta))(T\rightarrow S) := \{f:G(\cL)(T)\stackrel{\simeq}\rightarrow G(\delta)(T) \text{ such that } f|_{\GG_m(T)} = \id \},
$$
classifying theta structures on $(A,\cL)$, is representable by a finite \'{e}tale scheme over $S$.
\end{corollary}

\subsection{}

 Let now $U\subseteq \Aut_{\GG_m}G(\delta)$ be a subgroup and let $F_U$ be the sheafification for the \'{e}tale topology of the functor
$$
F_U(T) := U\backslash \mathrm{Isom}_{\GG_m}(G(\cL),G(\delta))(T).
$$

\begin{definition}
\label{definition:thetaStructureOfLevelU}
A {\em theta structure of level $U$} on $(A\rightarrow S,\cL)$ is an element of $F_U(S)$. 
\end{definition}

More explicitly, note that any theta structure $\beta$ (of full level) forgetfully induces a theta structure of level $U$ in $F_U(S)$. Conversely every theta structure of level $U$, locally on $S$, can be represented by a full theta structure $\beta$. Two theta structures of level $U$, represented by $\beta,\beta'$, respectively, are isomorphic if there exists $u\in U$  such that $\beta' = u\beta$.

\begin{proposition}
For any subgroup $U\subseteq \Aut_{\GG_m}G(\delta)$, $F_U$ is representable by a finite \'{e}tale scheme over $S$.
\end{proposition}

\begin{proof}
The sheaf $F_U$ is locally constant for the \'{e}tale topology, as follows from Cor. \ref{corollary:TorsorOfThetaStructuresIsFiniteEtale}. 
\end{proof}

Although the theory of theta structures of level $U$ can be developed along these general lines, we only work with one kind of level $U$, defined  as follows. Let
\begin{equation}
\label{equation:canonicalSplitting}
\sigma_{\rm{can}}(h) = (1,h), \quad h\in \underline{H}(\delta)\times \{0\},
\end{equation}
be the canonical maximal splitting of $G(\delta)$ over the subgroup $ \underline{H}(\delta) \simeq \underline{H}(\delta)\times \{0\}\subseteq \underline{K}(\delta)$. If a theta structure $\beta:G(\cL)\stackrel{\simeq}\rightarrow G(\delta)$ (of full level) is given, the data 
$$
(H,\sigma) := (\alpha^{-1}\underline{H}(\delta), \beta^{-1}\sigma_{\rm{can}}\alpha)
$$ 
is a splitting of $G(\cL)$ over $H$, and by Thm. \ref{theorem:DescentForLineBundles} the invertible sheaf $\cL$ descends to an invertible sheaf $\cM$ of degree 1 over $A/H$. In fact  we do not need a full theta structure to impose a descent datum on $(A,\cL)$, but only the part of it that preserves $(H,\sigma)$. This motivates the following definition: 

\begin{definition}
\label{definition:M-descent}
Let $A\rightarrow S$ be an abelian scheme of dimension $g$ over a scheme $S$ containing $1/d$, and let $\cL$ be a relatively ample, normalized invertible sheaf over $A$ of type $\delta$. An {\em $\cM$-descent structure} on $(A,\cL)$ is a theta structure of level
$$
U_0:=\{ u \in \Aut_{\GG_m}G(\delta) : u\,\sigma_{\rm{can}} \un{H}(\delta) \subseteq \sigma_{\rm{can}} \un{H}(\delta)\},
$$
the stabilizer in $\Aut_{\GG_m}G(\delta)$ of the canonical splitting $(\un{H}(\delta),\sigma_{\rm{can}})$.
\end{definition}

\subsection{} If $\cL$ is in addition symmetric, the concept of theta structure of level $U$ can be refined to take into account the extra  symmetry. To this end, consider the order two element of $\Aut_{\GG_m}G(\delta)$ given by 
\begin{align*}
D_{-1}: G(\delta) &\longrightarrow G(\delta) \\
(\lambda,z) &\longmapsto (\lambda,-z).
\end{align*}

\begin{definition}[\cite{Mumford:EqAbv2}, \S 6]
Let $A\rightarrow S$ be an abelian scheme of dimension $g$ over a scheme $S$ containing $1/2d$ and $\zeta_d$, and let $\cL$ be a symmetric, relatively ample, normalized invertible sheaf over $A$ of type $\delta$. A {\em symmetric} theta structure on $(A,\cL)$ is a theta structure $
\beta: G(\cL) \stackrel{\simeq}\rightarrow G(\delta),
$
such that
$$
D_{-1}\beta = \beta\delta_{-1},
$$
where $\delta_{-1}$ is the automorphism of $G(\cL)$ defined by \eqref{equation:delta-1}.
\end{definition}

Contrary to the case of theta structures, symmetric theta structures need not even exist locally. Indeed, to ensure existence over an \'{e}tale open we must require the theta characteristic  $e_*^{\cL}:A[2]\rightarrow \mu_2$ to be trivial on ${K(\cL)\cap A[2]}$ (\cite{Mumford:EqAbv1}, Prop. 2.3), which motivates the following:

\begin{definition}[\cite{Mumford:EqAbv2}, \S 6]
A symmetric invertible sheaf $\cL$ on $A$ is {\em totally symmetric} if $e_*^{\cL}\equiv 1$ on $A[2]$. 
\end{definition}

Equivalently, $\cL$ is totally symmetric if locally on $S$ it is the square of a symmetric invertible sheaf. Thus it is clear that if $\delta = (d_1,\ldots,d_g)$ is the type of a totally symmetric invertible sheaf $\cL$, then $2\mid d_i$ for all $i=1,\ldots,g$.

Consider now the \'{e}tale sheaf of groups 
$$
\Aut^{\rm{sym}}_{\GG_m} G(\delta)(T):= \{ \gamma \in \Aut_{\GG_m}G(\delta)(T): D_{-1}\gamma = \gamma D_{-1}\}.
$$
\begin{lemma}
\label{lemma:symmetricAutExactSequence} There is an exact sequence of \'{e}tale sheaves of groups 
$$
0 \rightarrow \underline{K}(\delta)[2] \stackrel{i}\rightarrow \Aut^{\rm{sym}}_{\GG_m}G(\delta) \stackrel{p}\rightarrow \Sp(\un{K}(\delta),\ppair{}{}_{\delta})\rightarrow 0,
$$
where the maps $i,p$ are as in Lemma \ref{lemma:AutExactSequence}.
\end{lemma}

\begin{proof}
As in Lemma \ref{lemma:AutExactSequence}, $i$ is injective and all elements of $\ker p$ must be of the form $i(z)(\lambda,z') = (\lambda\,\ppair{z}{z'}_{\delta},z')$, for some $z\in \underline{K}(\delta)$. Only now we also require that this isomorphism is symmetric. Thus
$$
\delta_{-1}i(z) : (\lambda,z')\mapsto (\lambda\,\ppair{z}{z'}_{\delta},-z')
$$
must equal 
$$
i(z)\delta_{-1} : (\lambda,z')\mapsto (\lambda\,\ppair{z}{-z'}_{\delta},-z').
$$
This equality holds if and only if $\ppair{z}{z'}_{\delta} \ppair{z}{-z'}_{\delta} =1$, which by bilinearity is equivalent to $\ppair{2z}{z'}_{\delta}=1$, for all $z'\in \underline{K}(\delta)$. Since $\ppair{}{}_{\delta}$ is non-degenerate, we must have $z\in \underline{K}(\delta)[2]$. For surjectivity of $p$, this follows from the theory of semi-characters, see \cite{ModularCorrespondences}, Lemma 15.
\end{proof}

Consequently, the \'{e}tale sheaf $\Aut^{\rm{sym}}_{\GG_m}G(\delta)$ is representable by a constant group scheme over $S$ and we also have:

\begin{corollary}
Let $\cL$ be totally symmetric. The $\Aut^{\rm{sym}}_{\GG_m}G(\delta)$-torsor
$$
\mathrm{Isom}^{\mathrm{sym}}_{\GG_m}(G(\cL),G(\delta))(T) := \{\beta\in  \mathrm{Isom}_{\GG_m}(G(\cL),\underline{G}(\delta))(T): \beta\delta_{-1} = D_1\beta \},
$$
classifying symmetric theta structures on $(A,\cL)$, is representable by a finite \'{e}tale scheme over $S$.
\end{corollary}

Symmetric theta structures of level $U\subseteq \Aut^{\rm{sym}}_{\GG_m} G(\delta)$ may then be defined analogously to Def. \ref{definition:thetaStructureOfLevelU}, via sheafification of the functor
$$
F^{\rm{sym}}_U(T) := U\backslash \mathrm{Isom}^{\mathrm{sym}}_{\GG_m}(G(\cL),G(\delta))(T),
$$
which is representable by a finite \'{e}tale scheme over $S$. In particular, we can refine the concept of an $\cM$-descent theta structure (Def. \ref{definition:M-descent}) as follows:

\begin{definition}
\label{definition:theta-descent}
Let $A\rightarrow S$ be an abelian scheme of dimension $g$ over a scheme $S$ containing $1/2d$ and $\zeta_d$, and let $\cL$ be a totally symmetric, relatively ample, normalized invertible sheaf over $A$ of type $\delta$. A {\em $\Theta$-descent structure} on $(A,\cL)$ is a symmetric theta structure of level
$$
U^{\rm{sym}}_0:=\{ u \in \Aut^{\rm{sym}}_{\GG_m}G(\delta) : u\,\sigma_{\rm{can}} \un{H}(\delta) \subseteq \sigma_{\rm{can}} \un{H}(\delta)\},
$$
the stabilizer in $\Aut^{\rm{sym}}_{\GG_m}G(\delta)$ of the canonical splitting $(\un{H}(\delta),\sigma_{\rm{can}})$ defined by \eqref{equation:canonicalSplitting}(which is symmetric, in the sense of Def. \ref{definition:symmetricSplitting}). 
\end{definition}

Intuitively, a $\Theta$-descent structure is the `weakest' type of structure that we must impose on $(A,\cL)$ in order to ensure the existence of an isogeny $\phi: A\rightarrow B$ together with an isomorphism $f:\phi^*\Theta \simeq \cL$, for some normalized symmetric invertible sheaf $\Theta$ inducing a principal polarization $\phi_{\Theta}: B\rightarrow B^t$.

\begin{remark}
It is possible to relax the condition of $\cL$ being totally symmetric with more general conditions on the parity of the type of $\cL$, as in \cite{BL}, Theorem 6.9.5, but for the sake of simplicity we will not pursue this avenue of generalization. 
\end{remark}

\subsection{} 
\label{subsection:ModuliStacks}
We now introduce the classifying stacks of some of the structures so far described, and briefly discuss some relationships among them. 

\begin{definition}
\label{definition:stackOfSymmetricSheaves}
Define $\sX_{g,d}$ to be the stack over $\ZZ[1/2d]$ whose category of sections $\sX_{g,d}(S)$ over a scheme $S \rightarrow \Spec(\ZZ[1/2d])$ is the groupoid of pairs $(A\rightarrow S, \cL)$ of an abelian scheme $A\rightarrow S$ of dimension $g$, together with a symmetric, relatively ample, normalized invertible sheaf $\cL$ of degree $d$.
\end{definition}

The stack $\sX_{g,d}$ is equipped with a natural map
\begin{align*}
\sX_{g,d} &\longrightarrow \sA_{g,d} \\
(A,\cL) &\longmapsto (A, \phi_{\cL}: A \rightarrow A^{t})
\end{align*}
to the moduli stack $\sA_{g,d}$ of abelian schemes with a polarization of degree $d$. For an abelian scheme $A$, the set of symmetric sheaves with fixed polarization is a torsor under $A[2]$, via the action on the theta characteristics. Thus $\sX_{g,d}$ is a torsor over $\sA_{g,d}$ under the 2-torsion of the universal abelian scheme. In particular:

\begin{proposition}
The algebraic stack $\sX_{g,d}$ is smooth over $\ZZ[1/2d]$. 
\end{proposition} 

\begin{proof}
Indeed $\sA_{g,d}[1/2d]$ is smooth over $\ZZ[1/2d]$ and $\sX_{g,d} \rightarrow \sA_{g,d}[1/2d]$ is \'{e}tale. 
\end{proof}

Next, fix a type $\delta = (d_1,\ldots,d_g)$ such that $d= d_1\cdots d_g$. Let $\sA_{g,\delta}$ be the open substack of $\sA_{g,d}$ given by pairs $(A,\phi)$ where $\phi$ is of type $\delta$. Note that $\sA_{g,d}[1/d] = \coprod_{\delta} \sA_{g,\delta}$, where $\delta$ ranges over all possible types. Assume additionally that $\delta$ is a type such that $2\mid d_i$ for all $i=1,\ldots, g$ and let $\sX^{\rm{tot}}_{g,\delta}$ be the stack classifying pairs $(A,\cL)$, with $\cL$ {\em totally} symmetric, normalized, relatively ample of type $\delta$. Then the forgetful map $\sX_{g,d} \rightarrow \sA_{g,d}$ above induces an identification
$$
\sX^{\rm{tot}}_{g,\delta} = \sA_{g,\delta} 
$$
since there is one and only one (normalized, relatively ample) totally symmetric sheaf of a given polarization of type $\delta$. We will use this identification throughout.

\begin{definition}
\label{definition:moduliStackOfThetaStructures}
Let $\delta$ be a type such that $2\mid d_i$ for all $i=1,\ldots, g$. Define $\sA_{g,\delta}(\Theta,\delta)$ to be the stack over $\ZZ[1/d,\zeta_d]$ ($d$ even) whose category of sections $\sA_{g,\delta}(\Theta,\delta)(S)$ over a scheme $S \rightarrow \Spec(\ZZ[1/d,\zeta_d])$ is the groupoid of triples $(A\rightarrow S, \cL, \beta)$ of an abelian scheme $A\rightarrow S$ of dimension $g$, together with a totally symmetric, relatively ample, normalized invertible sheaf $\cL$ of type $\delta$ and a symmetric theta structure $\beta:G(\cL) \stackrel{\simeq}\rightarrow G(\delta)$. 
\end{definition}

\begin{proposition}
\label{proposition:ModuliStackOfThetaStructuresIsAScheme}
The algebraic stack $\sA_{g,\delta}(\Theta,\delta)$ is smooth over $\ZZ[1/d,\zeta_d]$ ($d$ even). 
\end{proposition}

\begin{proof}
The stack $\sA_{g,\delta}(\Theta,\delta)$ is equipped with a natural map
$$
\sA_{g,\delta}(\Theta,\delta) \longrightarrow \sA_{g,\delta}(\delta) ,\quad
(A,\cL,\beta) \longmapsto (A, \cL, \alpha)
$$
to the moduli stack $\sA_{g,\delta}(\delta)$ of abelian schemes of dimension $g$ with a polarization of type $\delta$ and an arithmetic level structure $\alpha: K(\cL) \stackrel{\simeq}\rightarrow \underline{K}(\delta)$. Via Lemma \ref{lemma:symmetricAutExactSequence}, this map exhibits $\sA_{g,\delta}(\Theta,\delta)$ as a torsor over $\sA_{g,\delta}(\delta)$ under the 2-torsion of the universal abelian scheme, and is finite \'{e}tale over $\ZZ[1/d,\zeta_d]$ since $d$ is even. But $\sA_{g,\delta}(\delta)$ is smooth over $\ZZ[1/d,\zeta_d]$, so  $\sA_{g,\delta}(\Theta,\delta)$ is also smooth.
\end{proof}

\begin{remark}
\label{remark:quasi-projectiveness}
By \cite{Mumford:EqAbv2}, \S 6, the stack $\sA_{g,\delta}(\Theta,\delta)$ is representable by a quasi-projective scheme over $\ZZ[1/d,\zeta_d]$ if $4\mid d_i$, $i=1,\ldots,g$. 
\end{remark}

Note that by \cite{Mumford:EqAbv1}, Rmk. 2.3, 2.4, there is a finite \'{e}tale map $\sA_{g,\delta}(2\delta) \rightarrow \sA_{g,\delta}(\Theta,\delta)$, defined over $\ZZ[1/d,\zeta_{2d}]$, obtained by lifting an (arithmetic) level-2$\delta$ structure to a symmetric theta structure for $(A,\cL^{\otimes 2})$ and then projecting down by the map $\eta_2:G(\cL^{\otimes 2})\rightarrow G(\cL)$, defined by \eqref{equation:eta2}. The composition with the forgetful map $\sA_{g,\delta}(\Theta,\delta) \rightarrow \sA_{g,\delta}(\delta)$ is given by
$$
V: \sA_{g,\delta}(2\delta) \rightarrow \sA_{g,\delta}(\delta), \quad (A,\cL, \alpha) \mapsto (A,\cL, 2\alpha).
$$
Thus $\sA_{g,\delta}(\Theta,\delta)$ can be thought of as an `intermediate' moduli stack
\begin{equation}
\label{equation:thetaStructureSandwich}
\begin{tikzcd}
\sA_{g,\delta}(2\delta) \arrow[bend left=20]{rr}{V} \arrow{r}& \sA_{g,\delta}(\Theta,\delta) \arrow{r}& \sA_{g,\delta}(\delta),
\end{tikzcd}
\end{equation}
sitting in between the moduli stacks of full level -$\delta$ and -2$\delta$ structures.

We are now ready to study the moduli stack of $\Theta$-descent structures (Def. \ref{definition:theta-descent}): 

\begin{definition}
\label{definition:ModuliStackOfThetaDescentStructures}
Let $\delta$ be a type such that $2\mid d_i$ for all $i=1,\ldots, g$. Define $\sA_{g,\delta}(\Theta)$ to be the stack over $\ZZ[1/d]$ ($d$ even) whose category of sections $\sA_{g,\delta}(\Theta)(S)$ over a scheme $S \rightarrow \Spec(\ZZ[1/d])$ is the groupoid of triples $(A\rightarrow S, \cL, u_0)$ of an abelian scheme $A\rightarrow S$ of dimension $g$, together with a totally symmetric, relatively ample, normalized invertible sheaf $\cL$ of type $\delta$ and a $\Theta$-descent structure $u_0 \in F_{U_0^{\rm{sym}}}(S)$. 
\end{definition}

Let $u_0$ be a $\Theta$-descent structure, locally induced by a symmetric theta structure $\beta: G(\cL)\stackrel{\simeq}\rightarrow G(\delta)$. The maximal symmetric splitting $\beta^{-1}(\underline{H}(\delta),\sigma_{\rm{can}})$ of $G(\cL)$ over $H:= \alpha^{-1}(\underline{H}(\delta))$ is independent of the choice of local representative $\beta$ for $u_0$. Therefore we obtain a more concrete description of $\sA_{g,\delta}(\Theta)$: 

\begin{proposition}
\label{proposition:thetaDescentStackInTermsOfSubgroups}
$\sA_{g,\delta}(\Theta)$ can be identified with the stack over $\ZZ[1/d,\zeta_d]$ ($d$ even) whose category of sections $\sA_{g,\delta}(\Theta)(S)$ over a scheme $S \rightarrow \Spec(\ZZ[1/d,\zeta_d])$ is the groupoid of quadruples $(A\rightarrow S,\cL, H,\sigma)$, where $\cL$  is a relatively ample, normalized invertible sheaf of type $\delta$, $H\subseteq K(\cL)$ is a subgroup scheme locally isomorphic to $\underline{H}(\delta)$, and $(H,\sigma)$ is a maximal symmetric splitting of $G(\cL)$ over $H$.  
\end{proposition}

\begin{proposition}
The algebraic stack $\sA_{g,\delta}(\Theta)$ is smooth over $\ZZ[1/d]$ ($d$ even). 
\end{proposition} 

\begin{proof}
The forgetful map 
\begin{align*}
\sA_{g,\delta}(\Theta) &\longrightarrow \sA_{g,\delta}[1/d] \\
(A,\cL,H,\sigma) &\longmapsto (A, \cL)
\end{align*}
is finite \'{e}tale, since the \'{e}tale sheaf $F^{\rm{sym}}_{U_0}$ is locally constant. Thus the result follows by again noting that $\sA_{g,\delta}$ is smooth over $\ZZ[1/d]$. 
\end{proof}

If we denote by $\sA_{g,\delta,0}(\delta)$ the stack classifying triples $(A,\cL, H)$ as in Prop. \ref{proposition:thetaDescentStackInTermsOfSubgroups}, then there is a natural forgetful map 
$$
\sA_{g,\delta}(\Theta) \longrightarrow \sA_{g,\delta,0}(\delta), \quad (A,\cL,H,\sigma) \longmapsto (A,\cL,H).
$$
Less trivially, we also have

\begin{theorem}
\label{theorem:nontrivialFiniteEtaleCoverOfThetaDescentStack}
Let $\sA_{g,\delta,0}(2\delta)$ be the stack over $\ZZ[1/d]$ classifying pairs $(A,\cL)$ as in Prop. \ref{proposition:thetaDescentStackInTermsOfSubgroups}, together with a subgroup scheme $H\subseteq K(\cL^{\otimes 2})$ locally isomorphic to $\un{H}(2\delta)$. Then there is a finite \'{e}tale map $$ \sA_{g,\delta,0}(2\delta)\longrightarrow \sA_{g,\delta}(\Theta), $$
defined over $\ZZ[1/d]$. 
\end{theorem}

\begin{proof}
Let $(A,\cL)$ be as in Prop. \ref{proposition:thetaDescentStackInTermsOfSubgroups} and let $H\subseteq K(\cL^{\otimes 2})$ be a subgroup scheme locally isomorphic to $\un{H}(2\delta)$. Locally we may choose a symmetric theta structure $\beta: G(\cL^{\otimes 2})\rightarrow G(2\delta)$ lying over the given trivializing isomorphism $\alpha_0: H \simeq \un{H}(2\delta) = \un{H}(2\delta)\times \{0\} \subseteq \un{K}(2\delta)$. Choose a splitting $\sigma$ of $G(2\delta)$ over $\un{H}(2\delta)$. This splitting must necessarily be of the form
$
\sigma(h) = (\sigma_*(h),h,0),
$
for some homomorphism $\sigma_*: \un{H}(2\delta)\rightarrow \GG_m$. If we moreover require the splitting to be symmetric, i.e. $\delta_{-1}\sigma(h) = \sigma(-h)$, then we must have $\sigma_* = \sigma_*^{-1}$, i.e. $\sigma_*$ factors through a homomorphism $\sigma_*: \un{H}(2\delta)\rightarrow \mu_2$. Now the map $\eta_2:G(\cL^{\otimes 2})\rightarrow G(\cL)$ of \eqref{equation:eta2}, goes over via the symmetric theta structure $\beta$ to the map $H_2: G(2\delta)\rightarrow G(\delta)$, $(\lambda, z)\mapsto (\lambda^2, 2z)$(\cite{Mumford:EqAbv1}, \S 2). Therefore 
$$
H_2(\un{H}(2\delta),\sigma) = (\un{H}(\delta), \sigma_{\rm{can}}),
$$
and in particular the splitting $H_2(\un{H}(2\delta),\sigma)$ of $G(\delta)$ over $\un{H}(\delta)$ does not depend on the choice of $\sigma$. Thus the splitting $\eta_2(H,\beta^{-1}\sigma\alpha_0)$, a priori only defined locally, in fact descends to a well-defined maximal symmetric splitting $(H',\sigma')$ of $G(\cL)$ over $H'$. The desired map is then given by $(A,\cL,H) \mapsto (A,\cL, H',\sigma')$.  
\end{proof}

The composition of the map of Thm. \ref{theorem:nontrivialFiniteEtaleCoverOfThetaDescentStack} with the forgetful map $\sA_{g,\delta}(\Theta) \rightarrow \sA_{g,\delta,0}(\delta)$ is given by
$$
V: \sA_{g,\delta,0}(2\delta) \rightarrow \sA_{g,\delta,0}(\delta), \quad (A,\cL, H) \mapsto (A,\cL, H/H[2]),
$$
thus we may view $\sA_{g,\delta}(\Theta)$ as an `intermediate' moduli stack
\begin{equation}
\label{equation:thetaDescentStructureSandwich}
\begin{tikzcd}
\sA_{g,\delta,0}(2\delta) \arrow[bend left=20]{rr}{V} \arrow{r}& \sA_{g,\delta}(\Theta) \arrow{r}& \sA_{g,\delta,0}(\delta),
\end{tikzcd}
\end{equation}
where the maps of this commutative diagram are compatible with those of \eqref{equation:thetaStructureSandwich} with respect to the natural forgetful maps.  

\begin{example}
\label{example:GenusOne}
Suppose $g=1$, and let $\sM:=\sA_{1,1}$ be the moduli stack of elliptic curves. The totally symmetric, relatively ample, normalized invertible sheaves over an elliptic curve $E\rightarrow S$ are all of the form 
$$
\cL_m:= \cO_E(m\,e)\otimes\left(\Omega^1_{E/S}\right)^{\otimes m},
$$
of type $\delta = (m) \in 2\ZZ_{>0}$ (\cite{Polishchuck:Determinants}, \S 5.1). Since $\cL_m = \cL_1^{\otimes m}$, we clearly have an identification $\sM = \sA_{1,m}$, for all $m\in 2\ZZ_{>0}$. We also have $\sA_{1,m,0}(m) = \sM_0(m)$, where for any $N\in \ZZ_{>0}$ we denote by $\sM_0(N)$ the moduli stack of pairs $(E,H)$ of an elliptic curve and a subgroup scheme of $E[N]$, locally isomorphic to $\un{\ZZ/N\ZZ}$. The diagram \eqref{equation:thetaDescentStructureSandwich} thus specializes to
$$
\begin{tikzcd}
\sM_0(2m) \arrow[bend left=20]{rr}{V} \arrow{r}& \sA_{1,(m)}(\Theta) \arrow{r}& \sM_0(m).
\end{tikzcd}
$$
As is well-known, the degree of the finite \'{e}tale map $V$ in this case  is two. But the degree of the forgetful map $\sA_{1,m}(\Theta) \rightarrow \sM_0(m)$ is also two, since there are exactly two ways to symmetrically lift the cyclic subgroup $H(m)\subseteq K(m)$ to a subgroup of $G(m)$ (the number of ways being equal to the number of homomorphisms $\ZZ/m\ZZ\rightarrow \mu_2$). The map of Thm. \ref{theorem:nontrivialFiniteEtaleCoverOfThetaDescentStack} therefore induces an `accidental' isomorphism 
\begin{equation}
\label{equation:accidentalIsomorphism}
\sA_{1,m}(\Theta) \simeq \sM_0(2m)
\end{equation}
for all $m\in 2\ZZ_{>0}$. In general, this accidental isomorphism does not occur in higher degree. For example, if $\delta = (n,\ldots, n)\in \ZZ^g$, with $n>0$ even, then $\sA_{g,\delta}(\Theta) \rightarrow \sA_{g,\delta,0}(\delta)$ has degree $2^g$ but the map $V$ has degree $2^{\frac{g(g+1)}{2}}$.
\end{example}

Finally, there is also a natural, very important `descent' map 
\begin{equation}
\label{equation:DescentMap}
\begin{aligned}
\Des:\sA_{g,\delta}(\Theta) &\longrightarrow \sX_{g,1} \\
(A,\cL,H,\sigma) &\longmapsto (A/H, \Theta),
\end{aligned}
\end{equation}
where as usual $\Theta$ denotes the (relatively ample, normalized) symmetric invertible sheaf of degree 1 over $B=A/H$ corresponding to the maximal symmetric splitting $(H,\sigma)$ under Theorem \ref{theorem:DescentForSymmetricLineBundles}. Note that $\Theta$ is {\em not} totally symmetric, since it is of degree 1.

\subsection{Analytic Theory}
\label{section:analyticTheoryOfThetaStructures} Over the category of analytic spaces theta structures can be computed explicitly. We briefly recall how to do so for (full) theta structures, following \cite{BL}, \S 6, and then employ the same techniques to compute the `descent map' \eqref{equation:DescentMap} analytically. Therefore, let $A$ be an abelian variety of dimension $g$ over $\CC$, and suppose $A$ is uniformized as $A = V/\Lambda$, where $V$ is a complex vector space of of dimension $g$ and $\Lambda \subseteq V$ is a rank $2g$ lattice. Let $(\cdot,\cdot): V\times V \rightarrow \CC$ be a Hermitian form such that the alternating form $\psi = \im (\cdot,\cdot)$ is integral on $\Lambda$ and  suppose $\chi:\Lambda\rightarrow S^1$ satisfies $$\chi(\lambda_1 + \lambda_2)\chi(\lambda_1)^{-1}\chi(\lambda_2)^{-1} = e^{\pi i \psi(\lambda_1,\lambda_2)}.$$
Define the line bundle $L( (\cdot,\cdot),\chi)$ as the quotient of $V\times \CC$ by the action of $\Lambda$ given by
$$
\lambda(v,t) = (\chi(\lambda)e^{\pi (v,\lambda) + \frac{\pi}{2} (\lambda,\lambda)}, v + \lambda).
$$
By the Appel-Humbert Theorem (\cite{AbelianV}, \S 2), all line bundles  on $A$ are of this form. The condition of $L$ being ample is equivalent to the Hermitian form $(,)$ being positive definite (\cite{AbelianV}, Thm. of Lefschetz, App. to \S 2). Suppose then $L=L( (\cdot,\cdot),\chi)$ is ample, so that $\psi$ is non-degenerate, and choose a decomposition $\Lambda = \Lambda_1\oplus\Lambda_2$ into maximal isotropic spaces for $\psi$. Let $\mathrm{pr}: V \rightarrow V/\Lambda$ be the canonical projection, let 
$
\Lambda(L):= \mathrm{pr}^{-1}(K(L)) \subseteq V,
$
and let $\Lambda(L) = \Lambda(L)_1\oplus\Lambda(L)_2$ be the decomposition induced by that of $\Lambda$. Define the {\em covering theta group} of $L$ (\cite{BL}, \S 6.1)  to be the set
$$
\widetilde{G}(L):= \{(t,w): t\in\CC^{\times}, w\in\Lambda(L)\}
$$
with group operation given by
$
(t_1,w_1)(t_2,w_2) = (t_1t_2 e^{\pi (w_2,w_1)},w_1 + w_2).
$

\begin{proposition}[\cite{BL}, Thm. 6.1.3]
There is an isomorphism of central extensions 
$$
\begin{tikzcd}
1 \arrow{r} & \CC^{\times} \arrow{r} \arrow{d}{\mathrm{id}} &\widetilde{G}(L)/s_L(\Lambda) \arrow{r} \arrow{d} & \Lambda(L)/\Lambda\arrow{r} \arrow{d} & 1 \\
1 \arrow{r} & \CC^{\times}\arrow{r} &G(L) \arrow{r} & K(L) \arrow{r} & 1, 
\end{tikzcd}
$$
where $s_L: \Lambda \rightarrow \widetilde{G}(L)$ is the homomorphism given by 
$
\lambda \mapsto (\chi(\lambda)e^{\frac{\pi}{2} (\lambda,\lambda)},\lambda).
$
\end{proposition}

Suppose now that $L$ is also totally symmetric (so that $\chi\equiv 1$) and of type $\delta$. A full level structure $\alpha$ on $K(L) = \Lambda(L)/\Lambda$ can be lifted to a homomorphism $\widetilde{\alpha}: \Lambda(L) \rightarrow K(\delta)$ factoring through $\Lambda$. Let \begin{equation}
\begin{aligned}
\label{equation:distinguishedThetaStructure}
\widetilde{\beta}_0: \widetilde{G}(L) &\longrightarrow G(\delta)\\
(t,w) &\longmapsto ( t\,e^{-\pi i \psi(w_1,w_2) - \frac{\pi}{2}(w,w)}, \widetilde{\alpha}(w)),
\end{aligned}
\end{equation}
where we wrote $w = w_1 + w_2$ in terms of the chosen decomposition $\Lambda(L) = \Lambda(L)_1\oplus\Lambda(L)_2$. The map $\widetilde{\beta}_0$ is a homomorphism, inducing a symmetric theta structure $\beta_0: G(L) \stackrel{\simeq}\rightarrow G(\delta)$ lying over the given level structure $\alpha$ (\cite{BL}, Lemma 6.6.5). We call $\beta_0$ the {\em distinguished theta structure} on $(A,L,\alpha)$.

\subsection{}
\label{section:thetaDescentAnalyticTheory}
We now give analytic descriptions of the moduli stacks introduced in Sec. \ref{subsection:ModuliStacks}, and the maps between them (especially \eqref{equation:DescentMap}). For ease of exposition, we specialize to the case of elliptic curves (i.e. $g=1$) but the higher-dimensional cases can be similarly worked out using the analytic theory of theta structures of Sec. \ref{section:analyticTheoryOfThetaStructures}. 
Consider then the elliptic curve $E_{\tau} = \CC/\Lambda_{\tau}$, where $\Lambda_{\tau} = \langle \tau, 1 \rangle$, $\tau \in \mathfrak{h}$.  The ample totally symmetric line bundles are all of the form $L = L((\cdot,\cdot)_m, 1)$, $m\in 2\ZZ_{>0}$, where $(\cdot,\cdot)_m$ is the hermitian form associated to the complex structure $\tau$ and the symplectic form $\psi = \im (\cdot,\cdot)_m$ of matrix $\smalltwobytwo{0}{m}{-m}{0}$ with respect to the $\ZZ$-basis $\{\tau,1\}$ of $\Lambda_{\tau}$. Recall that the moduli stack $\sA^{\rm{an}}_{1,1}$ of elliptic curves over an analytic space (which we identify with the analytification of the smooth stack $\sA_{1,1}\otimes_{\ZZ}\CC$) is given by the analytic quotient stack $\left[\SL_2(\ZZ)\backslash\mathfrak{h}\right]$. Similarly, let $\sA^{\rm{an}}_{1,m} =\left(\sA_{1,m}\otimes_{\ZZ[1/m]}\CC\right)^{\rm{an}} $, $m\in 2\ZZ_{>0}$, be the moduli stack of elliptic curves over an analytic space, together with a polarization of type $(m)$.

\begin{proposition}
For any $m\in 2\ZZ_{>0}$, $\sA^{\rm{an}}_{1,m}$ is the analytic quotient stack $\left[\SL_2(\ZZ)\backslash\mathfrak{h}\right]$. 
\end{proposition}

\begin{proof}
By \cite{Mumford-GIT} \S 7.A, $\sA^{\rm{an}}_{1,m}$ is the analytic quotient stack $\left[\Gamma\backslash\mathfrak{h}\right]$, where $\Gamma = \{ M\in \GL_2(\ZZ) : M^tJ_mM = J_m\}$, where $J_m = \smalltwobytwo{0}{m}{-m}{0} = m\smalltwobytwo{0}{1}{-1}{0}$. Thus $\Gamma = \Sp_2(\ZZ) = \SL_2(\ZZ)$. Alternatively, this follows from the discussion of Ex. \ref{example:GenusOne}.   
\end{proof}

Next, consider the moduli stack $\sA_{1,m}(\Theta,m)$ of triples $(E,\cL,\beta)$ of Def. \ref{definition:moduliStackOfThetaStructures}.

\begin{proposition}
\label{proposition:thetaStructureModuliStack}
Let $m\in 2\ZZ_{>0}$, and let $\sA^{\rm{an}}_{1,m}(\Theta,m)$ be the analytification of the smooth stack $\sA_{1,m}(\Theta,m)\otimes_{\ZZ[1/m,\zeta_m]}\CC$. Then $\sA^{\rm{an}}_{1,m}(\Theta,m)$ is the analytic quotient stack $\left[\Gamma(m,2m)\backslash\mathfrak{h}\right]$, where 
$$
\Gamma(m,2m):= \{ \smalltwobytwo{a}{b}{c}{d} \in \SL_2(\ZZ): a,d\equiv 1\, (m)\text{ and } b,c \equiv 0 \,(2m) \}.
$$
\end{proposition}
\begin{proof}
Let   $
\sA_{1,2m}(2m)\rightarrow \sA_{1,m}(\Theta,m),
$
be the finite \'{e}tale surjective map appearing in \eqref{equation:thetaStructureSandwich}. The stack $\sA_{1,2m}(2m)$ is representable by a geometrically irreducible $\ZZ[1/m,\zeta_m]$-scheme, and  the analytification of its complex fiber is the Riemann surface $$\sA^{\rm{an}}_{1,2m}(2m) = \Gamma(2m)\backslash\mathfrak{h},$$ where $\Gamma(2m)$ is the principal congruence subgroup of level $2m$. By continuity, the analytification $\sA^{\rm{an}}_{1,m}(\Theta,m)$ must also be connected. Any connected finite cover of $\left[\SL_2(\ZZ)\backslash\mathfrak{h}\right]$ must be of the form $\left[\Gamma'\backslash\mathfrak{h}\right]$, for some subgroup $\Gamma'\subseteq\SL_2(\ZZ)$. The fact that $\Gamma' = \Gamma(m,2m)$ is stated without proof in \cite{Mumford-GIT}, \S 7.A, and is due to Igusa. We show how to obtain this non-trivial result since the same ideas will be used in the proof of Thm. \ref{theorem:Gamma0ModuliStack} below. Recall that $\Gamma(m)$ is the stabilizer of the full level $m$ structure given by 
\begin{align*}
\alpha_0: \frac{1}{m}\Lambda_{\tau}/\Lambda_{\tau} &\longrightarrow \ZZ/m\ZZ \times \ZZ/m\ZZ \\
\tau/m &\longmapsto (1,0) \\
1/m & \longmapsto (0,1),
\end{align*}
under the action of $\gamma\in\SL_2(\ZZ)$ given by $\gamma^T \mod m$. Similarly, $\Gamma'$ is the stabilizer of the distinguished theta structure $\beta_0$  \eqref{equation:distinguishedThetaStructure} induced by 
\begin{align*}
\widetilde{\beta}_0: \widetilde{G}(L((\cdot,\cdot)_m,1)) &\longrightarrow G(m)\\
(t,w) &\longmapsto (t\,e^{-\pi i m w_1 w_2 - \frac{\pi}{2}(w,w)_m}, \alpha_0(w))
\end{align*}
where we wrote $w = \tau w_1 + w_2 \in \frac{1}{m}\Lambda_{\tau}/\Lambda_{\tau}$. As stated above, for any $\gamma = \smalltwobytwo{a}{b}{c}{d}\in \SL_2(\ZZ)$ we have $\alpha_0(\gamma^T w) = \alpha_0(w)$ if and only if $\gamma \in \Gamma(m)$, the principal congruence subgroup. Moreover,
$$
m \gamma^T(w)_1\gamma^T(w)_2 = m(aw_1 + cw_2)(bw_1 + dw_2) = m(abw_1^2 + (ad+bc)w_1w_2 + cdw_2^2).
$$
Write $w_i = u_i/m \in \frac{1}{m}\Lambda_{\tau}$, with $u_i\in \ZZ$. Then
$$
e^{-\pi im(abw_1^2 + (ad+bc)w_1w_2 + cdw_2^2)} = e^{-\frac{2\pi i}{2m}( abu_1^2 + (ad+bc)u_1u_2 + cdu_2^2)}.
$$  
In order for this expression to be equal to $e^{-\pi i m w_1w_2}$ we need
$$
ab\equiv cd \equiv 0\, (2m), \quad ad + bc \equiv 1\, (2m).
$$
We know already that $a,d\equiv 1 (m)$ and that $b,c \equiv 0(m)$, since $\gamma \in \Gamma(m)$. Thus $a \equiv 1,m+1 (2m)$ and $b\equiv 0,m (2m)$. If $a\equiv 1 (2m)$ then the first equation is satisfied if $b\equiv 0 (2m)$. If  $a\equiv m+1 (2m)$ then either  $b\equiv 0(2m)$ or otherwise we have
$
(m+1)m \equiv 0\, (2m),
$
which is impossible, since $m+1$ is a unit in $2m$ ($m$ is even). Therefore $b\equiv 0 (2m)$ and the same holds for $c$. Since $ad-bc=1$ as integers, the last congruence holds as well. We conclude that the stabilizer of $(E_{\tau}, L((,)_m,1), \beta_0)$ is the group $\Gamma'=\Gamma(m,2m)$.
\end{proof}

\begin{remark}
Note that if $m\geq 4$ the action of $\Gamma(m,2m)$ on $\mathfrak{h}$ has  no non-trivial stabilizers, so that the quotient $\Gamma(m,2m)\backslash\mathfrak{h}$ is actually a Riemann surface. This is consistent with  Rmk. \ref{remark:quasi-projectiveness}.
\end{remark}

Finally, consider the analytification $\sA^{\rm{an}}_{1,m}(\Theta)$ of the moduli stack of Def. \ref{definition:ModuliStackOfThetaDescentStructures} and let 
$$
\Gamma_0(2m):= \{ \smalltwobytwo{a}{b}{c}{d} \in \SL_2(\ZZ): c \equiv 0\, (2m) \}.
$$

\begin{theorem}
\label{theorem:Gamma0ModuliStack}
Let $m\in 2\ZZ_{>0}$. The analytification $\sA^{\rm{an}}_{1,m}(\Theta)$ of $\sA_{1,m}(\Theta)\otimes_{\ZZ[1/m]}\CC$  is the analytic quotient stack $\left[\Gamma_0(2m)\backslash\mathfrak{h}\right]$. 
\end{theorem}
\begin{proof}
This follows from the `accidental' isomorphism \eqref{equation:accidentalIsomorphism}, but in the analytic category it can also be verified directly, as follows. The analytic stack $\sA^{\rm{an}}_{1,m}(\Theta)$ is connected, since is has a connected finite covering map $[\Gamma(m,2m)\backslash\mathfrak{h}] = \sA^{\rm{an}}_{1,m}(\Theta,m)\rightarrow \sA^{\rm{an}}_{1,m}(\Theta)$. Thus $\sA^{\rm{an}}_{1,m}(\Theta) = [\Gamma'\backslash \mathfrak{h}]$, where $\Gamma'\subseteq \SL_2(\ZZ)$ is a subgroup. To determine $\Gamma'$, consider again the distinguished theta structure \eqref{equation:distinguishedThetaStructure} 
$
\beta_0: G( L((\cdot,\cdot)_m,1)) \stackrel{\simeq}\rightarrow G(m).
$
Then $\Gamma'$ is the stabilizer of the maximal symmetric splitting $\beta_0^{-1}(H(m),\sigma_{\rm{can}})$ under the action of $\SL_2(\ZZ)$. Now $\alpha_0^{-1}(H(m)) = \langle 1/m \rangle \subseteq \frac{1}{m}\Lambda_{\tau}/\Lambda_{\tau}$, so $\gamma = \smalltwobytwo{a}{b}{c}{d}\in \SL_2(\ZZ)$ preserves this group if and only if $\gamma \in \Gamma_0(m)$ (recall that the action is via $\gamma^T$). The splitting $\sigma_0 = \beta_0^{-1}\sigma_{\rm{can}}\alpha_0$ lifts to a splitting 
$$
\widetilde{\sigma_0}(w) = (e^{\frac{\pi}{2} H(w,w)}, w) \in \widetilde{G}( L((\cdot,\cdot)_m,1))
$$  
of the covering theta group over $\frac{1}{m}\Lambda_{\tau}$. As computed in the proof of Prop. \ref{proposition:thetaStructureModuliStack}, the action of $\gamma$ changes $\beta_0(t,w)$ by a factor of $e^{-\pi i m \gamma^T(w)_1\gamma^T(w)_2}$, thus it changes $\sigma_0(w)$, $w = u/m\in \langle 1/m \rangle$, $u\in\ZZ$, by the factor
$
e^{-\frac{\pi i}{m}(cdu^2)} = e^{-\frac{2\pi i}{2m}(cdu^2)}. 
$
Thus $\gamma \in \Gamma'$ if and only if $cd\equiv 0\, (2m)$. Now we know already that $\gamma\in\Gamma_0(m)$, so that $c\equiv 0\, (m)$. Moreover $ad-bc =1$ as integers, so that $ad\equiv 1\, (m)$ and $d$ is a unit mod $m$, i.e. $(d,m)=1$. But then $(d,2m)=1$ as well, since $m$ is even, so that we must have $c\equiv 0 \, (2m)$, as required.   
\end{proof}

The natural forgetful maps $\sA^{\rm{an}}_{1,m}(\Theta,m)\rightarrow \sA^{\rm{an}}_{1,m}(\Theta)$ (resp. $\sA^{\rm{an}}_{1,m}(\Theta)\rightarrow \sA^{\rm{an}}_{1,m}$) correspond functorially to the  inclusion of fundamental groups $\Gamma(m,2m)\hookrightarrow \Gamma_0(2m)$ (resp. $\Gamma_0(2m)\rightarrow \SL_2(\ZZ)$). The computation of the map \eqref{equation:DescentMap}
$$
\mathrm{Des}: \sA^{\rm{an}}_{1,m}(\Theta) \longrightarrow \sX^{\rm{an}}_{1,1}
$$
is not as obvious. First, note that $\sX^{\rm{an}}_{1,1}$ is the disjoint union of two connected components $\sX^{+,\rm{an}}_{1,1},\sX^{-,\rm{an}}_{1,1}$, parametrizing the even/odd theta characteristics $e_*^{\Theta}$ (\cite{MoretBailly:PinceauxAbv}, VIII.3.2.5). In fact, we have
$$
\sX^{+,\rm{an}}_{1,1} = \left[\Gamma(1,2)\backslash\mathfrak{h}\right], \quad \sX^{-,\rm{an}}_{1,1} = \left[\SL_2(\ZZ)\backslash\mathfrak{h}\right]
$$
as analytic quotient stacks, where 
$$
\Gamma(1,2):= \{ \smalltwobytwo{a}{b}{c}{d} \in \SL_2(\ZZ): ab, cd \equiv 0\, (2) \}
$$
is the {\em theta group} (\cite{MoretBailly:PinceauxAbv}, VIII, 3.4.1.2, 3.4.2.2). But $\sA^{\rm{an}}_{1,m}(\Theta)$ is connected, thus the image of $\Des$ must lie in one of the $+,-$ components. 
\begin{theorem}
\label{theorem:DesMapFactorization}
The map $\Des$ factors as $\sA^{\rm{an}}_{1,m}(\Theta)\stackrel{\Des}\rightarrow \sX^{+,\rm{an}}_{1,1} \hookrightarrow \sX^{\rm{an}}_{1,1}$.
\end{theorem}

\begin{proof}
It suffices to compute the image under $\Des$ of the point $$(E_{\tau}, L=L((\cdot,\cdot)_m,1)),H_0=\langle 1/m \rangle, \sigma_0) \in \sA^{\rm{an}}_{1,m}(\Theta)(\CC),$$ that is, we need to compute the theta characteristic $e_*^{\Theta}$ of the descended line bundle $\Theta$ over $E_{\tau}/H_0$ corresponding to the descent data $(H_0,\sigma_0)$, where $\sigma_0$ is as in the proof of Thm. \ref{theorem:Gamma0ModuliStack}. From Sec. \ref{subsection:symmetricLineBundles}, $e_*^{\Theta}$ may be computed from the commutative diagram   
$$
\begin{tikzcd}
E_{\tau}/H_0[2] \arrow{rr}{s}\arrow{dr}{e_*^{\Theta}} & & G(\Theta^{\otimes 2})[2] = p^{-1}_{G(L^{\otimes 2})}(\epsilon_2(H_0)^{\perp})/\epsilon_2\sigma_0 (H_0)[2] \arrow{dl}{\eta_2}  \\
&\mu_2,  &
\end{tikzcd}
$$
where $s$ is any set-theoretic section. Now $G(\Theta^{\otimes 2})$ is a theta group over 
$$(\epsilon_2H_0)^{\perp}/\epsilon_2H_0\simeq \langle \tau/2\rangle \oplus \langle 1/2m\rangle / \langle 1/m \rangle = E_{\tau}/H_0[2],
$$
and its elements can be represented by pairs $(t,w)\in p^{-1}_{G(L^{\otimes 2})}(\epsilon_2(H_0)^{\perp})\subseteq G(L^{\otimes 2})$, with $w\in E_{\tau}/H_0[2]$. A lift $s$ is given by
$
s(w) = (e^{-\pi/2\,(w,w)_{2m}}, w), \quad w\in E_{\tau}/H_0[2].
$
The map $\eta_2$, on the other hand, is best computed by passing to $G(2m)$ via the symmetric theta structure $\beta_0$. The section $s$ goes over to the section
$$
\beta_0\circ s (w) = (e^{-\pi i\, 2m\, w_1w_2 - \pi (w,w)_{2m}}, \alpha_0(w)) \in G(2m)[2], \quad w = w_1\tau + w_2.
$$
The formula for $\eta_2: G(2m)\rightarrow G(m)$ can be found in \cite{Mumford:EqAbv2}, p. 317. We get
$$
\eta_2\circ\beta_0\circ s (w) = (e^{-2\pi i\, 2m\, w_1w_2 - 2\pi (w,w)_{2m}}, \alpha_0(2w)) = (e^{-2\pi i\, 2m\, w_1w_2 }, 0),
$$
the last equation following from the fact that $e^{2\pi (w,w)_{2m}} = e^{\pi/2 (2w,2w)_{2m}}\in s_{L^{\otimes 2}}(2w)$, and is thus trivial in the quotient $G(L^{\otimes 2}) = \widetilde{G}(L^{\otimes 2})/ s_{L^{\otimes 2}} \stackrel{\beta_0}\simeq G(2m)$. Writing $w_1 = u_1/2$ and $w_2 = u_2/2m$, we obtain 
\begin{equation}
\label{equation:Gamma0(2m)ThetaCharacteristic}
e_*^{\Theta}(\tau u_1/2 + u_2/2m) = \eta_2\circ\beta_0\circ s (w) = e^{-2\pi i\,2m\,\frac{u_1u_2}{4m}} = e^{\pi i \,u_1u_2},
\end{equation}
which is the standard {\em even} theta characteristic, as required. 
\end{proof}

By Thms. \ref{theorem:Gamma0ModuliStack} and \ref{theorem:DesMapFactorization}, the Des map corresponds by functoriality to a group homomorphism of fundamental groups
$$
\mathrm{Des}: \Gamma_0(2m) \longrightarrow \Gamma(1,2),
$$
which can be computed explicitly, as follows: 

\begin{lemma}
\label{lemma:DesMap}
For all $\smalltwobytwo{a}{b}{c}{d} \in \Gamma_0(2m)$, 
$$
\mathrm{Des}\smalltwobytwo{a}{b}{c}{d} = \smalltwobytwo{a}{b\,m}{c/m}{d} 
$$
\end{lemma}

\begin{proof}
The theta group $\Gamma(1,2)$ is the stabilizer of the action of the standard even theta characteristic
\begin{align*}
e^{+}_*: \frac{1}{2}\Lambda_{\tau}/\Lambda_{\tau} &\longrightarrow \mu_2 \\
w = u_1\tau/2 + u_2/2 &\longmapsto e^{\pi i u_1u_2} 
\end{align*}
with respect to the action of $\SL_2(\ZZ)$. On the other hand, $\Gamma_0(2m)$ is contained in the stabilizer of the theta characteristic  \eqref{equation:Gamma0(2m)ThetaCharacteristic}, as in the proof of Theorem  \ref{theorem:DesMapFactorization}. The $\Des$ homomorphism is then induced by the homomorphism of $\ZZ$-lattices
$$
\smalltwobytwo{1}{0}{0}{m}:\langle \tau/2 \rangle\oplus \langle 1/2m \rangle \longrightarrow \langle \tau/2 \rangle\oplus \langle 1/2 \rangle,
$$
intertwining the $\Gamma_0(2m)$-action on the left with the $\Gamma(1,2)$-action on the right. Unwinding the actions $\Des$ is given by
$$
\gamma = \smalltwobytwo{a}{b}{c}{d} \in \Gamma_0(2m) \longmapsto \left(\smalltwobytwo{1}{0}{0}{m}\gamma^T\smalltwobytwo{1}{0}{0}{1/m}\right)^T 
= \smalltwobytwo{a}{b\,m}{c/m}{d} \in \Gamma(1,2).
$$

\end{proof}

\section{Metaplectic stacks, theta multiplier bundles and half-forms}
\label{section:metaplecticStacks}

\subsection{}For any $g$-dimensional abelian scheme $\pi:A\rightarrow S$, let $\Omega^1_{A/S}$ be the sheaf of relative differential 1-forms and let
$$
\mf_{A/S}:= \det \pi_*\Omega^1_{A/S}
$$
be the Hodge bundle of $\pi$, an invertible sheaf over $S$.  For any base-change morphism
$$
\begin{tikzcd}
A \stackrel{\phi}\simeq A'\times_{S'}S \arrow{r}\arrow{d}   &   A'\arrow{d} \\
S \arrow{r}{\varphi} & S'
\end{tikzcd}
$$
of abelian schemes, there is a canonical $\cO_A$-module isomorphism $\Omega^1_{A/S} \simeq \phi^*\Omega^1_{A'\times S/S}$, which induces a canonical $\cO_S$-module isomorphism
$
\underline{\omega}_{A/S} \stackrel{\simeq}\longrightarrow \varphi^*{\underline{\omega}_{A'/S'}}.
$
In particular, the assignment
$$
\{ \pi: A\rightarrow S \} \longmapsto \underline{\omega}_{A/S}
$$
defines an invertible sheaf $\underline{\omega}$ on the stack $\sA_{g}:= \sA_{g,1}$ of principally polarized abelian schemes. This invertible sheaf is classified by a 1-morphism of algebraic stacks 
\begin{equation}
\label{equation:classifyingMapOfHodgeBundle}
\sA_{g} \stackrel{\mf}\longrightarrow B\GG_m,
\end{equation}
where $B\GG_m$ is the stack over $\ZZ$ whose groupoid of sections $B\GG_{m}(S)$ over a scheme $S$ is the category of invertible sheaves over $S$. 

Suppose now $1/2\in S$. A {\em square-root} of $\mf$ is an invertible sheaf $\mf^{1/2}$ such that $(\mf^{1/2})^{\otimes 2} \simeq \mf$. Equivalently, this is a collection of square-roots $(\mf_{A/S}^{1/2})^{\otimes 2} \simeq \mf_{A/S}$, for all abelian schemes $A\rightarrow S$, of formation compatible under base-change. In general, square-roots of $\mf$ do not exist (e.g. already for $g=1$, the element $[\mf]\in\Pic(\sA_{1}[1/6])$ is a generator \cite{Mumford:Picard}, so it cannot be divided by 2), but we can get around this issue by passing to $\mu_2$-gerbes. In particular, consider the squaring map $\cdot^2: \GG_m\rightarrow \GG_m$, which induces a 1-morphism of algebraic stacks
$$
\cdot^2 : B\GG_m \longrightarrow B\GG_m,
$$
taking an invertible sheaf $\mathcal{L}$ to $\mathcal{L}^{\otimes 2}$. 
\begin{definition}
The {\em metaplectic stack} $\sA_{g,1/2}$ is the 2-fiber product 
$$
\begin{tikzcd}
\sA_{g,1/2}:=\sA_g\times_{\underline{\omega},\cdot^2}B\GG_m \arrow{r} \arrow{d}{p_{1/2}}  & B\GG_m \arrow{d}{\cdot^2}\\
\sA_g \arrow{r}{\mf} & B\GG_m
\end{tikzcd}
$$
in the 2-category of algebraic stacks over $\ZZ[1/2]$, where the bottom horizontal arrow is given by \eqref{equation:classifyingMapOfHodgeBundle}. Equivalently, $\sA_{g,1/2}$ is the algebraic stack whose fiber over a scheme  $S\rightarrow \Spec(\ZZ[1/2])$ is the groupoid of triples $(\sA\rightarrow S, \phi_{\cL}, \mf^{1/2}_{A/S})$ of a principally polarized abelian scheme $(A,\phi_{\cL})$, together with a square-root $\mf^{1/2}_{A/S}$ of the Hodge bundle $\mf_{A/S}$.
\end{definition}

By definition $\sA_{g,1/2}$ is a $\mu_2$-gerbe over $\sA_g$. The top horizontal arrow
$
\sA_{g,1/2} \rightarrow B\GG_m
$
in the diagram defining $\sA_{g,1/2}$ determines an invertible sheaf $\underline{\omega}^{1/2}$ over the metaplectic stack $\sA_{g,1/2}$, canonically endowed with an isomorphism 
$$
p_{1/2}^*(\underline{\omega}) \simeq (\underline{\omega}^{1/2})^{\otimes 2},
$$
as invertible sheaves over $\sA_{g,1/2}$. Therefore $\mf^{1/2}$ is essentially a square-root of  $\mf$, albeit only defined over the $\mu_2$-gerbe $\sA_{g,1/2}$. This construction has the following cohomological interpretation. Consider the Kummer exact sequence of \'{e}tale sheaves over $\ZZ[1/2]$, 
\begin{equation}
\label{equation:KummerSequence}
1 \rightarrow \mu_2 \rightarrow \GG_m \stackrel{\cdot^2}\rightarrow \GG_m \rightarrow 1,
\end{equation}
to which there corresponds a long exact sequence 
$$
\ldots \rightarrow H^1_{\mathrm{\acute{e}t}}(\sA_g,\GG_m) \stackrel{\cdot^2}\rightarrow H^1_{\mathrm{\acute{e}t}}(\sA_g,\GG_m) \stackrel{\delta}\rightarrow H^2_{\mathrm{\acute{e}t}}(\sA_g,\mu_2) \rightarrow \ldots 
$$
in \'{e}tale cohomology. By identifying $H^1_{\mathrm{\acute{e}t}}(\sA_g,\GG_m)$ with $\Pic(\sA_g)$, the class $\delta(\underline{\omega})$ represents the obstruction to finding a square-root of $\underline{\omega}$. On the other hand, the group  $H^2_{\mathrm{\acute{e}t}}(\sA_g,\mu_2)$ can be identified with the group of $\mu_2$-gerbes over $\sA_g$. Under this identification, we have
$$
[\delta(\mf)] = [\sA_{g,1/2}] \in H^2_{\mathrm{\acute{e}t}}(\sA_g,\mu_2).
$$

\begin{remark}
If $\sX \rightarrow \sA_g$ is finite \'{e}tale, we also refer to the base-change $\sX\times_{\sA_g}\sA_{g,1/2}$ as a {\em metaplectic stack}. Intuitively, a `metaplectic stack' should be any moduli stack of abelian schemes (and some attendant structures) endowed with the canonical square-root of the Hodge bundle $\un{\omega}^{1/2}$.  
\end{remark}

\subsection{} We now briefly recall the theory of {\em theta multiplier bundles}, which have been introduced and studied in detail in \cite{C1}. Suppose $1/2,\zeta_4\in S$ and let $\Theta$ be a symmetric, relatively ample, normalized invertible sheaf of degree 1 over $A$. Locally on $S$, there are isomorphisms
$$
(K(\Theta^{\otimes 4})=A[4], e_{\Theta^4}) \stackrel{\simeq}\rightarrow (\underline{\ZZ/4\ZZ}^{2g}, \ppair{}{}_4),
$$
where $\ppair{}{}_4$ is the standard symplectic pairing \eqref{equation:standardSymplecticPairing} of type $\delta = (4,\ldots, 4)\in \ZZ^g$. The `reduction modulo 2'  of $(K(\Theta^{\otimes 4}), e_{\Theta^4})$ is the symplectic 2-group $(K(\Theta^{\otimes 2})=A[2],e_{\Theta^4})$, locally isomorphic to $(\underline{\ZZ/2\ZZ}^{2g}, \ppair{}{}_2)$. Additionally, the group scheme $K(\Theta^{\otimes 2})$ is  endowed with the theta characteristic $e_*^{\Theta}:A[2]\rightarrow \mu_2$, which is a quadratic form with associated bilinear form equal to $e_{\Theta^{\otimes 2}}$ (\cite{Mumford:EqAbv2}, Cor. 2.1). The triple $(K(\Theta^{\otimes 4}), e_{\theta^4}, e_*^{\Theta})$ is called a {\em symplectic 4-group scheme with theta characteristic} (\cite{C1} Def. 3.1). Let
$$
I_{A,\Theta}(T\rightarrow S):=\underline{\mathrm{Isom}}( (K(\Theta^{\otimes 4})(T), e_{\Theta^4}, e_*^{\Theta}), (\underline{\ZZ/4\ZZ}^{2g}(T), \ppair{}{}_4, e^{\pm}_*))
$$ 
be the functor that sends an $S$-scheme to the set of symplectic isomorphisms $\psi$ between $K(\Theta^{\otimes 4})(T)$ and $\underline{\ZZ/4\ZZ}^{2g}(T)$ such that the reduction modulo 2 of $\psi$ sends $e_*^{\Theta}$ to either the {\em standard even theta characteristic}, i.e. the quadratic form  defined on the standard symplectic basis by 
$$
e_*^+(x_1,\ldots,x_g, y_1, \ldots, y_g) = (-1)^{\sum_{i=1}^g x_iy_i},
$$
or a similarly defined standard odd theta characteristic $e_*^-$, according to whether $e_*^{\Theta}$ is even or odd. The functor $I_{A,\Theta}$ is a finite \'{e}tale $\underline{\Gamma}(2g,\pm)$-torsor, where $\Gamma(2g,\pm)$ is the group of automorphisms of the symplectic 4-group with standard theta characteristic $(\ZZ/4\ZZ^{2g} , \ppair{}{}_4, e^{\pm}_*)$. This is a group extension 
$$
1 \rightarrow \Gamma(2g,2) \rightarrow \Gamma(2g,\pm) \stackrel{\mod 2}\longrightarrow \mathrm{O}(2g,\pm) \rightarrow 1,
$$
where $\Gamma(2g,2) = \{ \gamma \in \Sp(2g,\ZZ/4\ZZ): \gamma \equiv I_{2g} \mod 2\}$ and $\mathrm{O}(2g,\pm)$ is the orthogonal group of $(\ZZ/2\ZZ^{2g}, e_*^{\pm})$.  

As $\Gamma(2g,\pm)\subseteq \Sp(2g,\ZZ/4\ZZ)$, the restriction of the determinant homomorphism is trivial. However, as  shown in \cite{C1}, \S 2, there exists a finer {\em discriminant} homomorphism (whose construction is due to Deligne, private communication), which can be characterized as follows: 

\begin{theorem}[\cite{C1}, Thm. 2.4]
\label{theorem:discriminant}
For all $g\geq 1$, there is a unique homomorphism 
$$
\lambda:\underline{\Gamma}(2g,\pm)\rightarrow \mu_4,
$$
called the `discriminant' of $(\underline{\ZZ/4\ZZ}^{2g} , \ppair{}{}_4, e^{\pm}_*)$, such that: 
\begin{itemize}
\item[(i)] $\lambda$ fits in the commutative diagram 
$$
\begin{tikzcd}
1 \arrow{r} & \un{\Gamma}(2g,2) \arrow{r} \arrow{d}{e_*^{\pm}} & \un{\Gamma}(2g,\pm) \arrow{r} \arrow{d}{\lambda} & \un{\mathrm{O}}(2g,\pm) \arrow{r} \arrow{d}{D} & 1 \\
1 \arrow{r} & \mu_2\arrow{r} &\mu_4 \arrow{r}{\cdot^2} & \mu_2 \arrow{r} & 1, 
\end{tikzcd}
$$
where $D:\mathrm{O}(2g,\pm)  \rightarrow \mu_2$ is the Dickson invariant, and $e_*^{\pm}:\Gamma(2g,2)\rightarrow \mu_2$ is the linearization of the quadratic form $e_*^{\pm}$ on $\Gamma(2) \simeq S^2(\ZZ/2\ZZ^2)$ (symmetric 2-tensors). 
\item[(ii)] For one (all) $v\in \un{\ZZ/4\ZZ}^2$ such that $e_*^{\pm}(v \mod 2) = -1$, we have 
$$
\lambda(t_v) = \zeta_4,
$$ 
where $t_v$ is the anisotropic transvection $t_v(z) = z + \ppair{v}{z}_4v$.  
\end{itemize}
\end{theorem}

The theory of discriminants can be employed to the define {\em theta multiplier bundles}. Recall (Def. \ref{definition:stackOfSymmetricSheaves}) that $\sX_{g,1}$ is the classifying stack of pairs $(A,\Theta)$ of symmetric, normalized, relatively ample invertible sheaves $\Theta$ of degree 1 over an abelian scheme $A$. We henceforth view it as a stack over $\ZZ[1/2,\zeta_4]$.  

\begin{definition}[\cite{C1}, Def. 3.2]
The {\em theta multiplier bundle} $\cM_{\Theta}$ is the invertible sheaf over $\sX_{g,1}$ associated to the $\mu_4$-torsor given by the functor 
$$
(A,\Theta)  \longmapsto \lambda_*I_{A,\Theta}.
$$
\end{definition}

Just as for the Hodge bundle $\un{\omega}$ over  $\sA_g$, we can extract a square-root of $\cM_{\Theta}$ by considering the fiber product of algebraic stacks
\begin{equation}
\label{equation:tildeStack}
\begin{tikzcd}
\widetilde{\sX}_{g,1}:=\sX_{g,1}\times_{\cM_{\Theta},\cdot^2}B\GG_m \arrow{r}{\cM_{\Theta}^{1/2}} \arrow{d} & B\GG_m \arrow{d}{\cdot^2}\\
\sX_{g,1} \arrow{r}{\cM_{\Theta}} & B\GG_m,
\end{tikzcd}
\end{equation}
a $\mu_2$-gerbe over $\sX_{g,1}$ representing the class $\delta[\cM_{\Theta}] \in H^2_{\mathrm{\acute{e}t}}(\sX_{g,1},\mu_2)$, where $\delta$ is the connecting homomorphism associated to the Kummer sequence \eqref{equation:KummerSequence}.

\subsection{} Consider the natural forgetful map
$$
f: \sX_{g,1} \longrightarrow \sA_g.
$$
By the functoriality of Hodge bundles, there is a canonical isomorphism $f^*\un{\omega} = \un{\omega}_{\sX_{g,1}}$, thus we omit unnecessary subscripts and always denote the Hodge bundle by $\un{\omega}$, whenever the base is clear from context. The map $f$ induces a commutative diagram in \'{e}tale cohomology 
$$
\begin{tikzcd}
H^1_{\mathrm{\acute{e}t}}(\sA_g,\GG_m)\arrow{r}{\delta}\arrow{d}{f^*}  & H^2_{\mathrm{\acute{e}t}}(\sA_g,\mu_2)\arrow{d}{f^*}\\
H^1_{\mathrm{\acute{e}t}}(\sX_{g,1},\GG_m)\arrow{r}{\delta} & H^2_{\mathrm{\acute{e}t}}(\sX_{g,1},\mu_2),
\end{tikzcd}
$$
where $\delta$ is the connecting homomorphism coming from the Kummer sequence \eqref{equation:KummerSequence}. In terms of gerbes, the class $f^*\delta[\un{\omega}]$ then corresponds to the metaplectic stack $\sX_{g,1/2}:= \sA_{g,1/2}\times_{\sA_g} \sX_{g,1}$, canonically endowed with a square root $\un{\omega}^{1/2}$, which is the pull-back of the square-root of $\un{\omega}$ over $\sA_{g,1/2}$. We denote by $\widetilde{\sX_{g,1}}\times^{\mu_2}_{\sX_{g,1}}{\sX_{g,1/2}}\stackrel{p}\rightarrow\sX_{g,1}$ the contracted $\mu_2$-gerbe product, a $\mu_2$-gerbe over $\sX_{g,1}$ whose class in $H^2_{\mathrm{\acute{e}t}}(\sX_{g,1},\mu_2)$ corresponds to the product $\delta[\cM_{\Theta}\otimes\mf] = \delta[\cM_{\Theta}]\cdot \delta[\mf]$.

\begin{theorem}
\label{theorem:MetaplecticCorrection}
There is a canonical trivialization
$$
s_{\Theta}: \sX_{g,1} \longrightarrow \widetilde{\sX_{g,1}}\times^{\mu_2}_{\sX_{g,1}}{\sX_{g,1/2}}
$$
of the contracted $\mu_2$-gerbe product $\widetilde{\sX_{g,1}}\times^{\mu_2}_{\sX_{g,1}}{\sX_{g,1/2}}\stackrel{p}\longrightarrow\sX_{g,1}$.
\end{theorem}
\begin{proof}
We first show that the class of the $\mu_2$-gerbe $\widetilde{\sX_{g,1}}\times^{\mu_2}_{\sX_{g,1}}{\sX_{g,1/2}}$ is trivial in $H^2_{\mathrm{\acute{e}t}}(\sX_{g,1},\mu_2)$ which is equivalent to showing that $\delta[\cM_{\Theta}\otimes\mf]=0$, i.e. that $\cM_{\Theta}\otimes\mf$ is a square in $\Pic(\sX_{g,1})$. To see this, consider the invertible sheaf $\cJ_{g,1}$ over $\sX_{g,1}$ defined by the functor
$$
\{\pi:A\rightarrow S, \Theta\} \longmapsto \pi_*\Theta.
$$ 
By the `key formula' (\cite{C1}, Thm 5.1), there is an isomorphism 
$$
\cJ^{-\otimes 2}_{g,1} \simeq \cM_{\Theta}\otimes\mf,
$$
which immediately implies $\delta[\cM_{\Theta}\otimes\mf]=0$. But the key formula also provides a canonical trivialization of the $\mu_2$-gerbe $\widetilde{\sX_{g,1}}\times^{\mu_2}_{\sX_{g,1}}{\sX_{g,1/2}}$, i.e. the section $s_{\Theta}$ of $p$ induced by $\cJ_{g,1}^{-1}$. Explicitly, if we view $\widetilde{\sX_{g,1}}\times^{\mu_2}_{\sX_{g,1}}{\sX_{g,1/2}}$ as the sheaf of groupoids whose section above a scheme $S$ is the groupoid of triples $(A,\Theta, \cL)$, where $\cL$ is a square-root of $\cM_{\Theta}\otimes\un{\omega}$, then the section $s_{\Theta}$ is given by the explicit formula 
$
s_{\Theta}(A,\Theta) = (A,\Theta, \pi_*\Theta^{-1}).
$
\end{proof}

The formation of the invertible sheaf $\cM_{\Theta}^{1/2}\otimes\mf^{1/2}$ makes sense over $\widetilde{\sX_{g,1}}\times^{\mu_2}_{\sX_{g,1}}{\sX_{g,1/2}}$. By Thm. \ref{theorem:MetaplecticCorrection}, we may then define 
$$
\un{\omega}_{\Theta}^{1/2}:= s_{\Theta}^*\left(\cM_{\Theta}^{1/2}\otimes\mf^{1/2}\right),
$$
the unique square-root of $\cM_{\Theta}\otimes\mf$ such that there is an isomorphism 
$
\un{\omega}_{\Theta}^{1/2} \stackrel{\simeq}\longrightarrow \cJ_{g,1}^{-1}
$ (as explained in \cite{C1}, \S 6.1, this isomorphism, which is essentially unique, is the algebraic analog of the functional equation of Riemann's theta function). 

\begin{definition}
\label{definition:half-forms}
The invertible sheaf $\underline{\omega}_{\Theta}^{1/2}$ over $\sX_{g,1}$ is called the {\em bundle of half-forms}. For all $k\in\ZZ$ we denote by $\mf_{\Theta}^{k/2}$ the tensor power $(\mf_{\Theta}^{1/2})^{\otimes k}$.
\end{definition}

A consequence of Thm. \ref{theorem:MetaplecticCorrection} is that the classes $\delta[\cM_{\Theta}]^{-1} = \delta[\cM_{\Theta}]$ and $\delta[\un{\omega}]$ are the same in $H^2_{\mathrm{\acute{e}t}}(\sX_{g,1},\mu_2)$, and that there is a canonical isomorphism of $\mu_2$-gerbes given by
\begin{align*}
\widetilde{\sX_{g,1}} &\stackrel{\simeq}\longrightarrow \sX_{g,1/2} \\
(\pi:A\rightarrow S,\Theta, \cM_{\Theta}^{1/2}) &\longmapsto (A,\Theta,\pi_*\Theta^{-1}\otimes \cM_{\Theta}^{-1/2}),
\end{align*}
where we view $\widetilde{\sX_{g,1}}$ as the sheaf of groupoids of fiber the pair $(A,\Theta)$ together with a square-root of the theta multiplier bundle, and similarly for the metaplectic stack $\sX_{g,1/2}$ with respect to the Hodge bundle. In particular, the canonical square root $\cM_{\Theta}^{1/2}$ over $\widetilde{\sX_{g,1}}$ goes over to a square root defined over the metaplectic stack, and we have:
\begin{corollary}
\label{corollary:metaplecticCharacter}
There is a canonical square root $\cM_{\Theta}^{1/2}$ defined over the metaplectic stack $p_{1/2}:\sX_{g,1/2}\rightarrow \sX_{g,1}$, such that there is an isomorphism
$$
\cM_{\Theta}^{1/2}\otimes\un{\omega}^{1/2} \simeq p_{1/2}^*\left(\un{\omega}^{1/2}_{\Theta}\right),
$$
i.e. the product $\cM_{\Theta}^{1/2}\otimes\un{\omega}^{1/2}$ descends to the bundle of half-forms over $\sX_{g,1}$. 
\end{corollary}

Vice-versa, the square root $\un{\omega}^{1/2}$ goes over to an invertible sheaf over $\widetilde{\sX_{g,1}}$. In other words, we have the following

\begin{principle} A square-root of the Hodge bundle of $(A,\Theta)$ determines in a unique way a square-root of the theta multiplier bundle, and vice-versa.
\end{principle}

\subsection{} The theory of half-forms, together with the theory of $\Theta$-descent structures of Sec. \ref{section:Theta-Descent}, combine to give an algebro-geometric definition of {\em modular forms of half-integral weight}, in the sense of Shimura \cite{Shimura}, and natural higher-dimensional generalizations. In particular, let $\sA_{g,\delta}(\Theta)$ be the classifying stack of $\Theta$-descent structures on pairs $(A,\cL)$ of a $g$-dimensional abelian scheme together with a totally symmetric sheaf of type $\delta$ (Def. \ref{definition:ModuliStackOfThetaDescentStructures}), and let 
$$
\Des: \sA_{g,\delta}(\Theta) \longrightarrow \sX_{g,1}
$$
be the natural descent map \eqref{equation:DescentMap}. 

\begin{definition}
\label{definition:GeneralModForms}
Let $k\in \ZZ$. A {\em modular form of half-integral weight} $k/2$, degree $g$ and type $\delta$ is a global section of the invertible sheaf 
$$
\cL_{g,\delta}^{k/2} := \Des^*\un{\omega}_{\Theta}^{k/2},
$$ 
where $\un{\omega}^{1/2}_{\Theta}$ is the bundle of half-forms (Def. \ref{definition:half-forms}).
\end{definition}

\begin{remark}
The stack $\sX_{g,1}$ decomposes into two connected components $\sX^+_{g,1}\coprod\sX^-_{g,1}$, according to the parity of the theta characteristic (\cite{MoretBailly:PinceauxAbv}, VIII.3.4). Therefore Def. \ref{definition:GeneralModForms} effectively contains two types of very different objects, according to whether $\un{\omega}^{1/2}_{\Theta}$ is first restricted to the even or odd component.  
\end{remark}

\begin{remark}
Def. \ref{definition:GeneralModForms} does not impose any `growth conditions' at infinity, an issue that surely deserves consideration elsewhere, and that can likely be approached via the theory of cubical torsors (\cite{MoretBailly:PinceauxAbv}, I). For example, for $g=1$ growth conditions can be imposed by extending the line bundles $\cL_{1,(m)}^{k/2}, m\in 2\ZZ_{>0},$ over the appropriate Tate curve. 
\end{remark}

The `simplest' instance of Def. \ref{definition:GeneralModForms}(degree 1, type $\delta = (2)$), already gives a new, geometric definition of Shimura's classical modular forms of half-integral weight \cite{Shimura}:

\begin{definition}
\label{definition:ShimuraModularForm}
Let $k\in \ZZ$. A {\em (Shimura) modular form of half-integral weight} $k/2$ is a global section of the invertible sheaf 
$$
\cL_{1,(2)}^{k/2} := \Des^*\un{\omega}_{\Theta}^{k/2},
$$ 
defined over $\sA_{1,(2)}(\Theta) = \sM_0(4)[1/2,\zeta_4]$ (cf. \eqref{equation:accidentalIsomorphism}). 
\end{definition}

Shimura modular forms of half-integral weight and higher level $N$, $4\mid N$, can be similarly defined by pulling back $\cL_{1,(2)}^{k/2}$ under the natural forgetful maps
$$
\sM_0(N) \longrightarrow \sM_0(4).
$$
Type and level should not be confused. For example, Shimura modular forms of half-integral weight and level $N$ are {\em not} the same as  modular forms of half-integral weight, degree 1 and type $N/2$ (i.e. sections of $\cL_{1,(N/2)}^{k/2}$), even though they are both defined over the modular stack $\sM_0(N) = \sA_{1,(N/2)}(\Theta)$.

\begin{remark}
The problem of defining an algebraic theory of Shimura modular forms of half-integral weight has been previously considered by Nick Ramsey in \cite{Ramsey}(see the Introduction to this article for a discussion comparing the two definitions). 
\end{remark}

Finally, Def. \ref{definition:GeneralModForms} can also be phrased in a `stack-free' way in terms of test objects, in a manner similar to that of \cite{Katz}, Def. 1.1 for classical modular forms of integral weight.  

\begin{definition}[Katz-style]
Let $k\in \ZZ$. A {\em modular form of half-integral weight $k/2$, degree $g$ and type $\delta$} is a rule $f$ which assigns to each quadruple $(A\rightarrow S,\cL, H,\sigma)$ of an abelian scheme $A$ over a scheme $S\rightarrow \Spec(\ZZ[1/d,\zeta_{4}])$, a totally symmetric, relatively ample, normalized invertible sheaf $\cL$ of type $\delta$, together with a maximal symmetric splitting $\sigma$ of $G(\cL)$ over $H$, a section $f(A,\cL, H,\sigma)$ of $\un{\omega}_{A,\Theta}^{k/2}$ (the $k$-th tensor power of the bundle of half-forms over $A$ determined by $\Theta$, the descended symmetric sheaf of degree 1 over $A/H\rightarrow S$ corresponding to the descent data $(H,\sigma)$), such that 
\begin{itemize}
\item[1.] $f(A,\cL, H,\sigma)$ only depends on the isomorphism class of $(A,\cL, H,\sigma)$.
\item[2.] The formation of $f$ is compatible under base-change $S'\rightarrow S$. 
\end{itemize}

\end{definition}

\subsection{Analytic Theory} We now compute the bundle of half-forms analytically in the case $g=1$ and even theta characteristic. The other cases (higher degree and/or odd theta characteristic) can be similarly worked out. Consider again the analytic quotient stack $\sA^{\rm{an}}_{1,1} = [\SL_2(\ZZ)\backslash \mathfrak{h}]$, classifying elliptic curves over an analytic space. The projection map $\mathrm{pr}:\mathfrak{h}\rightarrow [\SL_2(\ZZ)\backslash \mathfrak{h}]$ classifies the universal framed elliptic curve $$\sE:= \CC\times \mathfrak{h}/\Lambda\rightarrow \mathfrak{h},\quad \Lambda = \{(z,\tau)\in \CC\times\mathfrak{h} : z\in \ZZ + \tau\ZZ\}.$$
Isomorphism classes of invertible sheaves over the stack quotient $[\SL_2(\ZZ)\backslash \mathfrak{h}]$ are in 1-1 correspondence with 1-cocycle classes in $H^1(\SL_2(\ZZ), \cO^{\times}_{\mathfrak{h}})$, since $\mathfrak{h}$ is a Stein manifold. In particular, the Hodge bundle $\un{\omega}$ over $\sA^{\rm{an}}_{1,1}$ admits a trivialization
$
dz: \pr^*\un{\omega} \stackrel{\simeq}\longrightarrow \cO_{\mathfrak{h}},
$
whose corresponding 1-cocycle in $Z^1(\SL_2(\ZZ), \cO^{\times}_{\mathfrak{h}})$ is given by
$$
j(\un{\omega}): \smalltwobytwo{a}{b}{c}{d} \longmapsto c\tau + d.
$$
Let $\Mp_2(\ZZ)$ be the {\em metaplectic group}, the unique non-trivial central extension
$$
1 \rightarrow \mu_2 \rightarrow \Mp_2(\ZZ) \rightarrow \SL_2(\ZZ) \rightarrow 0.
$$ 
As is well-known, there is no faithful representation of $\Mp_2(\ZZ)$ into a matrix group. Rather, elements of $\Mp_2(\ZZ)$ can be represented as pairs $(\gamma, \phi)$, with $\gamma=\smalltwobytwo{a}{b}{c}{d}\in\SL_2(\ZZ)$ and $\phi \in \cO_{\mathfrak{h}}^{\times}$ satisfying  $\phi^2 = c\tau + d$. The multiplication in $\Mp_2(\ZZ)$ is then given by the rule
$$
(\gamma_1, \phi_2)(\gamma_2, \phi_2) = (\gamma_1\gamma_2, \phi_1(\gamma_2\tau)\phi_2(\tau)).
$$  
The group $\Mp_2(\ZZ)$ acts on $\mathfrak{h}$ via the action of the quotient $\SL_2(\ZZ)$, and we have, essentially by definition:
\begin{proposition}
The analytification $\sA^{\rm{an}}_{1,1/2}$ of $\sA_{1,1/2}\otimes \CC$ is the analytic quotient stack  $[\Mp_2(\ZZ)\backslash \mathfrak{h}]$.
\end{proposition}

Isomorphism classes of invertible sheaves over the stack quotient $[\Mp_2(\ZZ)\backslash \mathfrak{h}]$ are also in 1-1 correspondence with 1-cocycle classes in $H^1(\Mp_2(\ZZ), \cO^{\times}_{\mathfrak{h}})$. In particular, the line bundle $\pr^*\un{\omega}^{1/2}$ admits a trivialization
\begin{equation}
\label{equation:half-trivialization}
dz^{1/2}: \pr^*\un{\omega}^{1/2} \stackrel{\simeq}\rightarrow \cO_{\mathfrak{h}}
\end{equation}
whose corresponding 1-cocycle in $Z^1(\Mp_2(\ZZ), \cO^{\times}_{\mathfrak{h}})$ is given by
$
\left(\gamma, \phi\right) \longmapsto \phi. 
$

Next, recall (Sec. \ref{section:thetaDescentAnalyticTheory}) that $\sX^{+,\rm{an}}_{1,1} = [\Gamma(1,2)\backslash \mathfrak{h}]$, so that we have:

\begin{proposition}
\label{proposition:AnalyticMetaplecticThetaStack}
The analytification $\sX^{+,\rm{an}}_{1,1/2}$ of the metaplectic stack $\sX^{+}_{1,1/2}\otimes\CC$ is the analytic quotient stack  $[\Mp\Gamma(1,2)\backslash \mathfrak{h}]$, where  $\Mp\Gamma(1,2)$ is the inverse image of $\Gamma(1,2)\subseteq \SL_2(\ZZ)$ in $\Mp_2(\ZZ)$. The 1-cocycle in $Z^1(\Mp\Gamma(1,2),\cO_{\mathfrak{h}}^{\times})$ corresponding to $\un{\omega}^{1/2}$ and the trivialization \eqref{equation:half-trivialization} is the restriction of $\left(\gamma, \phi\right) \longmapsto \phi$ to $\Mp\Gamma(1,2)$.
\end{proposition}

The map $\pr: \mathfrak{h} \rightarrow [\Gamma(1,2)\backslash \mathfrak{h}]$ now classifies the pair $(\sE, Th)$, where $Th$ is the symmetric line bundle of degree 1 over $\sE\rightarrow \mathfrak{h}$ given by the divisor of the 2-variable Riemann theta function $\sum_{n\in\ZZ} e^{\pi i \tau n^2 + 2\pi i n z}$. The theta multiplier bundle $\cM(Th) = \pr^*\cM_{\Theta}$ is a trivial $\mu_4$-torsor over $\mathfrak{h}$, and it admits a trivialization $c: \pr^*\cM_{\Theta} \stackrel{\simeq}\longrightarrow \mu_4$ via a constant section $c$ such that the corresponding 1-cocycle in $Z^1(\Gamma(1,2),\mu_4)$ is the character
$$
\lambda^{-1}: \Gamma(1,2) \longrightarrow \mu_4.
$$
Here $\lambda$ is the extension of the discriminant (Thm. \ref{theorem:discriminant})  of $(\ZZ/4\ZZ^2, \ppair{}{}_4, e_*^+)$ to $\Gamma(1,2)$ via the reduction mod 4 map $\Gamma(1,2)\rightarrow \Gamma(2,+)$. Consider then the central extension 
$$
1 \rightarrow \mu_2 \rightarrow \widetilde{\Gamma}(1,2) \rightarrow \Gamma(1,2) \rightarrow 1,
$$ 
where  $\widetilde{\Gamma}(1,2)$ is the group of pairs $(\gamma, \mu) \in \Gamma(1,2)\times \mu_8$ such that $\mu^2 = \lambda(\gamma)^{-1}$, with multiplication given by $(\gamma_1, \mu_1)(\gamma_2,\mu_2) = (\gamma_1\gamma_2, \mu_1\mu_2)$. We then have: 

\begin{proposition}
\label{proposition:AnalyticTildeStack}
The analytification $\widetilde{\sX}^{+,\rm{an}}_{1,1}$ of $\widetilde{\sX}^{+}_{1,1}\otimes\CC$ is the analytic quotient stack  $[\widetilde{\Gamma}(1,2)\backslash \mathfrak{h}]$. The line bundle $\cM_{\Theta}^{1/2}$ over $\widetilde{\sX}^{+,\rm{an}}_{1,1}$ admits a trivialization $c^{1/2}$ over $\mathfrak{h}$ such that its corresponding 1-cocycle in $Z^1(\widetilde{\Gamma}(1,2),\cO_{\mathfrak{h}}^{\times})$ is give by the character $(\gamma, \mu) \mapsto \mu$. 
\end{proposition}

Finally, let $\widehat{\Gamma}(1,2)$ be the Baer sum of the two extensions $\Mp\Gamma(1,2)$ and $\widetilde{\Gamma}(1,2)$. Explicitly, this is the set of pairs $(\gamma = \smalltwobytwo{a}{b}{c}{d}, \psi) \in \Gamma(1,2)\times\mathcal{O}^{\times}_{\mathfrak{h}}$ with $\psi^2 = \lambda^{-1}(\gamma)(c\tau + d)$, and with multiplication given by
$
(\gamma_1, \psi_1)(\gamma_2, \psi_2) = (\gamma_1\gamma_2, \psi_1(\gamma_2\tau)\psi_2(\tau)).
$
By Props. \ref{proposition:AnalyticMetaplecticThetaStack} and \ref{proposition:AnalyticTildeStack}, we have
$$
\widetilde{\sX}^{+,\rm{an}}_{g,1}\times^{\mu_2}_{\sX^{+,\rm{an}}_{g,1}}{\sX^{\rm{an}}_{g,1/2}} = [\widehat{\Gamma}(1,2)\backslash \mathfrak{h}],
$$
and the line bundle $\cM_{\Theta}^{1/2}\otimes \un{\omega}^{1/2}$ corresponds to the 1-cocycle 
\begin{equation}
\label{equation:thetaHalfCocycle}
j(\cM_{\Theta}^{1/2}\otimes \un{\omega}^{1/2}): (\gamma, \psi) \longmapsto \psi
\end{equation}
with respect to the given trivializations of Props. \ref{proposition:AnalyticMetaplecticThetaStack} and \ref{proposition:AnalyticTildeStack}. By Thm. \ref{theorem:MetaplecticCorrection} this $\mu_2$-gerbe over $\sX^{+,\rm{an}}_{g,1}$ is trivial, i.e. the central extension $\widehat{\Gamma}(1,2)$ is split.

\begin{theorem}
The canonical trivialization $s_{\Theta}$ of Thm. \ref{theorem:MetaplecticCorrection} corresponds to the  section 
$$
s_{\Theta}: \Gamma(1,2) \longrightarrow \widehat{\Gamma}(1,2)
$$
given by $
s_{\Theta}(\gamma) = (\gamma, \vartheta(\gamma \tau)\vartheta(\tau)^{-1})
$, where $\vartheta(\tau) = \sum_{n\in \ZZ} e^{\pi i n^2 \tau}$ is the (1-variable) Riemann-Jacobi theta function.
\end{theorem} 
\begin{proof}
To compute the trivialization $s_{\Theta}$ of Thm. \ref{theorem:MetaplecticCorrection} explicitly, note that the line bundle $\cJ_{g,1}^{-1}$ over $\sX^{+,\rm{an}}_{g,1}$ admits a trivialization
$$
\vartheta(\tau): \pr^*\cJ_{g,1}^{-1} \stackrel{\simeq}\longrightarrow \cO_{\mathfrak{h}}
$$
by the Riemann-Jacobi theta function (which is non-vanishing on $\mathfrak{h}$). The corresponding 1-cocycle is given by
$
\gamma \mapsto \vartheta(\gamma \tau)\vartheta(\tau)^{-1}
$
which gives the formula for $s_{\Theta}$ above (note that by the classical functional equation of $\vartheta(\tau)$, it follows that for all $\gamma = \smalltwobytwo{a}{b}{c}{d} \in \Gamma(1,2)$, 
$$
\vartheta(\gamma \tau)^2\vartheta(\tau)^{-2} = \lambda(\gamma)^{-1}(c\tau+d)
$$
and the map $s_{\Theta}$ is indeed a splitting of $\widehat{\Gamma}(1,2)\rightarrow \Gamma(1,2)$). 
\end{proof}

In particular, the 1-cocycle in $Z^1(\Gamma(1,2), \cO_{\mathfrak{h}}^{\times})$ corresponding to the bundle of half-forms $\un{\omega}^{1/2}_{\Theta} = s_{\Theta}^*(\cM_{\Theta}^{1/2}\otimes\un{\omega}^{1/2})$ (with respect to the given trivializations of $\cM_{\Theta}^{1/2}$ and $\un{\omega}^{1/2}$ in Props. \ref{proposition:AnalyticMetaplecticThetaStack}, \ref{proposition:AnalyticTildeStack}) is given by
\begin{equation}
\label{equation:cocycleOfHalf-Forms}
j(\un{\omega}_{\Theta}^{1/2}) = j(\cM_{\Theta}^{1/2}\otimes \un{\omega}^{1/2})\circ s_{\Theta}: \gamma \mapsto \vartheta(\gamma \tau)\vartheta(\tau)^{-1}.
\end{equation}

Alternatively, the cocycle of half-forms can be computed using Cor. \ref{corollary:metaplecticCharacter}, as follows. Let 
$$
j(\un{\omega}^{1/2}): (\gamma, \phi) \longmapsto \phi
$$
be the cocycle corresponding to the line bundle $\un{\omega}^{1/2}$ over the metaplectic stack $\sX^{+,\rm{an}}_{g,1/2} = [\Mp\Gamma(1,2)\backslash \mathfrak{h}]$ with respect to the trivialization \eqref{equation:half-trivialization}. Cor. \ref{corollary:metaplecticCharacter} then translates into:   

\begin{proposition}
\label{proposition:tildeCharacter}
There exists a unique homomorphism
$$
\widetilde{\lambda}: \Mp\Gamma(1,2) \longrightarrow \mu_8
$$
such that $\widetilde{\lambda}^{-1}\cdot j(\un{\omega}^{1/2})(\gamma,\phi)(\tau) =\vartheta(\gamma \tau) \vartheta(\tau)^{-1} = j(\un{\omega}_{\Theta}^{1/2})$.
\end{proposition}

The proposition is enough to compute $\widetilde{\lambda}$ explicitly. For example, 
$
\widetilde{\lambda}(\smalltwobytwo{0}{-1}{1}{0},\sqrt{\tau}) = \sqrt{-i}
$, since $\widetilde{\lambda}^{-1}\cdot j(\un{\omega}^{1/2})(\smalltwobytwo{0}{-1}{1}{0},\sqrt{\tau}) = \widetilde{\lambda}^{-1}(\smalltwobytwo{0}{-1}{1}{0},\sqrt{\tau})\sqrt{\tau} = \sqrt{i}\sqrt{\tau}$.

\subsection{} Next, we want to compute analytically the sections of $\cL_{1,(2)}^{k/2}$ (Def. \ref{definition:ShimuraModularForm}) over $\sA_{1,(2)}(\Theta)$, i.e. Shimura modular forms of half-integral weight, and show that they coincide with the classical definition of \cite{Shimura}. Consider first the map
$$
\sM^{\rm{an}}_0(2) = \left[\Gamma_0(2)\backslash\mathfrak{h}\right] \stackrel{V}\longrightarrow \sA^{\rm{an}}_{1,1} = \left[\SL_2(\ZZ)\backslash\mathfrak{h}\right]
$$
which is {\em not} the forgetful map, but rather the map induced at the modular level by $(E, H) \mapsto E/H$, where $H\subseteq E$ is a subgroup locally isomorphic to $\un{\ZZ/2\ZZ}$. As in the proof of Lemma \ref{lemma:DesMap}, this map corresponds by functoriality to a map of fundamental groups
$$
\smalltwobytwo{a}{b}{c}{d} \in \Gamma_0(2) \stackrel{V_{\pi_1}}\longmapsto \smalltwobytwo{a}{2b}{c/2}{d}\in \SL_2(\ZZ).
$$   
The map $V$ can also be lifted to the universal covering space $\mathfrak{h}$, and we have

\begin{proposition}
\label{proposition:V-holo}
The map $V$ lifts to the holomorphic map $\mathfrak{h}\rightarrow \mathfrak{h}$ given by $\tau \mapsto 2\tau$. 
\end{proposition}
\begin{proof}
Denote by $\sE(1)$ the universal elliptic curve over $\sA^{\rm{an}}_{1,1}$. This is the quotient stack $\CC\times\mathfrak{h}/ \Lambda\ltimes \SL_2(\ZZ)$, where the action of $\gamma = \smalltwobytwo{a}{b}{c}{d}\in \SL_2(\ZZ)$ is given by $(\frac{z}{c\tau + d}, \gamma\tau)$. Similarly, denote by  $\sE_0(2)$ the universal elliptic curve over $\sM^{\rm{an}}_0(2)$, which is the same quotient only restricted to $\Gamma_0(2)$. The map $V$ lifts to a holomorphic map $V_{\mathfrak{h}}(\tau) \in \cO_{\mathfrak{h}}$ such that the map
$
(z,\tau) \longrightarrow (z, V_{\mathfrak{h}}(\tau))
$
must descend to an equivariant morphism $\sE_0(2) \rightarrow \sE(1)$ lying over $V$, i.e it must satisfy
$$
\left(\frac{z}{c\tau + d}, V_{\mathfrak{h}}(\gamma\tau)\right)   = V_{\pi_1}(\gamma)(z, V_{\mathfrak{h}}(\tau))
$$
for all $\gamma = \smalltwobytwo{a}{b}{c}{d}\in \Gamma_0(2)$. A simple computation shows that $V_{\mathfrak{h}}(\tau) = 2\tau$ is the only such map.
\end{proof}

Next, consider the descent map $\Des: \sM^{\rm{an}}_0(4) \rightarrow \sX_{1,1}^{\rm{an}}$. By Thm. \ref{theorem:DesMapFactorization} we know that this map factors as 
$$
\sM^{\rm{an}}_0(4) = \left[\Gamma_0(4)\backslash\mathfrak{h}\right] \stackrel{\Des}\longrightarrow \sX_{1,1}^{+,\rm{an}} = \left[\Gamma(1,2)\backslash\mathfrak{h}\right],
$$
so it suffices to only consider the even component of $\sX_{1,1}^{\rm{an}}$. Note that $\Des$ fits in the commutative diagram 
$$
\begin{tikzcd}
\sM^{\rm{an}}_0(4)\arrow{r}{\Des}\arrow{d}{}  &\sX_{1,1}^{+,\rm{an}} \arrow{d} \\
\sM^{\rm{an}}_0(2)\arrow{r}{V} & \sA_{1,1}^{\rm{an}},
\end{tikzcd}
$$
where the vertical arrows are the natural forgetful maps corresponding by functoriality to the inclusion of the corresponding fundamental groups. Therefore, the holomorphic map $\Des_{\mathfrak{h}} \in \cO_{\mathfrak{h}}$ induced by $\Des$ on the universal covering space is also given by $\tau \rightarrow 2\tau$, as in Prop. \ref{proposition:V-holo}. 

\begin{theorem}
\label{theorem:AnalyticShimuraModularFormsCoincide}
Let $k \in \ZZ$. The line bundle $\cL_{1,(2)}^{k/2} := \Des^*\un{\omega}^{k/2}_{\Theta}$ over $\sM^{\rm{an}}_0(4)$ corresponds (with respect to the same trivializations over $\mathfrak{h}$ as in \eqref{equation:cocycleOfHalf-Forms}) to the 1-cocycle in $Z^1(\Gamma_0(4), \cO_{\mathfrak{h}}^{\times})$ given by
$$
j_{k/2}: \gamma \longmapsto \left(\vartheta_2(\gamma\tau)\vartheta^{-1}_2(\tau)\right)^k, 
$$
where $\vartheta_2(\tau) = \sum_{n\in\ZZ}e^{2\pi i n^2 \tau} = \vartheta(2\tau)$ is the level 2 theta constant.   
\end{theorem}
\begin{proof}
By Lemma \ref{lemma:DesMap} we know that $\Des$ acts on fundamental groups by 
$$
\smalltwobytwo{a}{b}{c}{d}\in \Gamma_0(4) \stackrel{\Des_{\pi_1}}\longrightarrow \smalltwobytwo{a}{2b}{c/2}{d} \in \Gamma(1,2),
$$
and by the above discussion we know that it acts on the universal covering space $\mathfrak{h}$ by $\tau \mapsto 2\tau$. Combined with the explicit formula for the 1-cocycle of the half-forms given by \eqref{equation:cocycleOfHalf-Forms}, we get
$$
j(\cL_{1,(2)}^{1/2})(\gamma)(\tau) = j(\un{\omega}^{1/2}_{\Theta})(\Des_{\pi_1}\gamma)(2\tau) = \vartheta(\Des_{\pi_1}\gamma(2\tau))\vartheta(2\tau)^{-1} = \vartheta(2\gamma \tau)\vartheta(2\tau)^{-1},
$$
as required. 
\end{proof}

It follows that the analytic global sections of $\cL_{1,(2)}^{k/2}$ can be identified with holomorphic functions $f:\mathfrak{h}\rightarrow \CC$ such that
$$
f(\gamma\tau) = j_{k/2}(\gamma) f(\tau),\quad \gamma \in \Gamma_0(4),
$$
which is precisely the definition of \cite{Shimura} (modulo growth conditions at the cusps). 

\begin{remark}
The same arguments as in Thm. \ref{theorem:AnalyticShimuraModularFormsCoincide} show that the analytic sections of $\cL_{1,(m)}^{k/2}$, for $m \in 2\ZZ_{>0}$, i.e. modular forms of half-integral weight of degree 1, type $(m)$ and weight $k/2$, transform according to the $k$-th power of the $\Gamma_0(2m)$-cocycle given by the level $m$ theta constant $\vartheta_m(\tau) = \sum_{n\in\ZZ}e^{m\pi i n^2 \tau} = \vartheta(m\tau)$. Note that these are different from Shimura modular forms of level $2m$, which are obtain by restricting $j_{k/2}$ to $\Gamma_0(2m) \subseteq \Gamma_0(4)$. 
\end{remark}

Alternatively, the cocycle (i.e. automorphy factor) of modular forms of half-integral weight can also be computed using the character $\widetilde{\lambda}$ of Prop. \ref{proposition:tildeCharacter}, and it is equal to
$$
j_{k/2}(\gamma)(\tau) = \widetilde{\lambda}^{-k}\,j(\un{\omega}^{k/2})(\Des_{\pi_1}\gamma, 2\tau).
$$

\begin{remark}
It would be interesting to work out the analytic theory of general modular forms of half-integral weight, degree $g$ and type $\delta$ (Def. \ref{definition:GeneralModForms}), whose explicit definition as analytic functions can be easily worked out by generalizing the arguments above. 
\end{remark}

\section{Weil Bundles}
\label{section:WeilBundles}

\subsection{}  Let $G=G(\delta)$ be the Heisenberg group scheme of type $\delta = (d_1,\ldots, d_g), d = d_1\cdots d_g$, over a scheme $S$ containing $1/d$. Let $\cV$ be a {\em representation} of $G(\delta)$, i.e. a locally free sheaf of finite rank over $S$ together with a homomorphism $\rho: G\rightarrow \Aut_{\cO_S}(\cV)$. We then have an action of $\GG_m \subseteq G$ on $\cV$, and thus a decomposition $\cV \simeq \oplus_{i\in \ZZ} \cV^{(i)}$, where $\GG_m$ acts via $\lambda \mapsto \lambda^i$ on each component $\cV^{(i)}$. We say that $\cV$ has {\em weight $i$} whenever $\cV = \cV^{(i)}$. By the following `Stone-Von Neumann'-type Theorem (\cite{MoretBailly:PinceauxAbv}, Thm. V.2.4.2) we have

\begin{theorem}
\label{theorem:SVN}
Let $\cV,\cV'$ be representations of $G$ of weight one, locally free of rank $d$ as $\cO_S$-modules. Then the $G$-equivariant  morphism 
$$
\cV\otimes_{\cO_S}Hom_{G}(\cV,\cV')\stackrel{\simeq}\longrightarrow \cV',\quad v\otimes u \longmapsto u(v) 
$$
where $Hom_{G}(\cV,\cV')$ is given the trivial action, is an isomorphism.  
\end{theorem}

Suppose now that $(H,\sigma)$ is a splitting of $G$ over $H\subseteq K=K(\delta)$, and let $\overline{G} = p^{-1}H^{\perp}/\sigma H$. By \cite{MoretBailly:PinceauxAbv}, Thm. V.3.2 (i), the functor
\begin{align*}
\mathrm{Rep}^{(1)}(G) &\longrightarrow \mathrm{Rep}^{(1)}(\overline{G})\\
\cV &\longmapsto \cV^{\sigma H}
\end{align*}
is an equivalence of categories between weight one representations of $G$ and $\overline{G}$. For any two $\cV,\cV'$ as in Thm. \ref{theorem:SVN} above, we then have an isomorphism 
$$
Hom_{G}(\cV,\cV') \simeq Hom_{\overline{G}}(\cV^{\sigma H},\cV'^{\sigma H}).
$$
If $(H,\sigma)$ is maximal, then $\overline{G} = \GG_m$ and $ \cV^{\sigma H},\cV'^{\sigma H}$ are invertible $\cO_S$-modules (\cite{MoretBailly:PinceauxAbv} proof of Thm. V.3.2 (i)), so that 
\begin{equation}
\label{equation:invertibleH-Invariants}
Hom_{G}(\cV,\cV') = Hom_{\cO_S}(\cV^{\sigma H},\cV'^{\sigma H}) = (\cV^{\sigma H})^{-1}\otimes\cV'^{\sigma H}.
\end{equation}
Next, we construct an explicit weight one representation $\cV(\delta)$ of $G = G(\delta)$. As a sheaf, we let $\cV(\delta)$ be the free $\cO_S$-module of functions $f: H(\delta) \rightarrow \cO_S$. This free module has a canonical basis of delta functions $\{\delta_{\nu}\}_{\nu \in H(\delta)}$, characterized by $\delta_{\nu}(\mu) = \delta_{\nu\mu}$.

\begin{definition}
\label{definition:SchrodingerRep}
The {\em Schr\"{o}dinger representation} of $G=G(\delta)$ is the $\cO_S$-module $\cV(\delta)$ of functions $f: H(\delta) \rightarrow \cO_S$, with $G$-action given by
$$
\rho(\lambda,x,y)\, f(y') = \lambda \langle x,y' \rangle_{\delta} f(y'+y).
$$
\end{definition}
Now since $H(\delta)$ is of rank $d$, the representation $\cV(\delta)$ is also of rank $d$ and it is of weight one by definition. Thus $\cV(\delta)$ satisfies the hypothesis of Thm. \ref{theorem:SVN}, and we may deduce the following special case of Thm. \ref{theorem:SVN} (cp. \cite{MoretBailly:PinceauxAbv}, Thm. V.3.4)

\begin{theorem}
\label{theorem:SVNWithMaximalSplitting}
Let $(\un{H}(\delta),\sigma_{\rm{can}})$ be the canonical maximal splitting of $G(\delta)$, defined by \eqref{equation:canonicalSplitting}. Let $\cV$ be a representation of $G(\delta)$ of weight one, locally free of rank $d$ as a $\cO_S$-module. Then there exists a canonical $G(\delta)$-equivariant  isomorphism
$$
\cV(\delta)\otimes_{\cO_S}\cV^{\sigma_{\rm{can}} \un{H}(\delta)}\stackrel{\simeq}\longrightarrow \cV, 
$$
where the action on $\cV^{\sigma_{\rm{can}} \un{H}(\delta)}$ is trivial.
\end{theorem}

\begin{proof}
This follows from  Thm. \ref{theorem:SVN} and \eqref{equation:invertibleH-Invariants}, and by application of the canonical trivialization $\cV(\delta)^{\sigma_{\rm{can}} H(\delta)} \simeq \cO_S\cdot \delta_0$. 
\end{proof}

\subsection{} 
\label{section:idealTheoremSection}
Suppose now $2\mid d_i$, $i=1,\ldots, g$, and let $(\pi:A\rightarrow S,\cL)$ be an abelian scheme together with a totally symmetric, normalized, relatively ample invertible sheaf $\cL$ of type $\delta$. The pushforward $\pi_*\cL$ is locally free of rank $d$, of formation compatible under base-change, and it is a weight one representation of the theta group $G(\cL)$ (\cite{Mumford:EqAbv2}, \S 6). Passing to the universal case, let $\cJ_{g,\delta}$ be the sheaf on $\sA_{g,\delta}$ given by the functor
\begin{equation}
\label{equation:JacobiFormsSheaf}
\cJ_{g,\delta}: (\pi:A\rightarrow S,\cL) \longmapsto \pi_*\cL,
\end{equation}
a weight one representation of the theta group of the universal abelian scheme (stack) over $\sA_{g,\delta}$. Next, consider the algebraic stack  $\sA_{g,\delta}(\Theta,\delta)$ classifying pairs $(A,\cL)$ as above together with a symmetric theta structure $\beta: G(\cL) \stackrel{\simeq}\rightarrow G(\delta)$. By a slight abuse of notation, we let
$$
\Des: \sA_{g,\delta}(\Theta,\delta) \longrightarrow \sX_{g,1}
$$
be the descent map \eqref{equation:DescentMap}, which really factors as $\sA_{g,\delta}(\Theta,\delta)\rightarrow \sA_{g,\delta}(\Theta) \stackrel{\Des}\rightarrow \sX_{g,1}$. Also by abusing notation, denote again by $\cJ_{g,\delta}$ the pull-back of \eqref{equation:JacobiFormsSheaf}  to $\sA_{g,\delta}(\Theta,\delta)$. Over this stack, we may also form the Schr\"{o}dinger representation $\cV(\delta)$ and compare it to $\cJ_{g,\delta}$, as in Thm. \ref{theorem:SVNWithMaximalSplitting} above:

\begin{theorem}
\label{theorem:IdealTheorem}
There is an isomorphism
$$
\cV(\delta)\otimes \Des^*\un{\omega}^{-1/2}_{\Theta} \stackrel{\simeq}\longrightarrow  \cJ_{g,\delta},
$$
as locally free sheaves of rank $d$ over $\sA_{g,\delta}(\Theta,\delta)$,
canonically defined up to $\pm 1$.
\end{theorem}

\begin{proof}
First, note that for any triple $(A,\cL,\beta)$, we have a canonical maximal splitting $(H,\sigma)$ of $G(\cL)$ over $H:= \alpha^{-1}\un{H}(\delta)$ given by pulling back $\sigma_{\rm{can}}$ via $\beta$ (this is how the forgetful map $\sA_{g,\delta}(\Theta,\delta)\rightarrow \sA_{g,\delta}(\Theta)$ is defined). By the proof of Thm. \ref{theorem:DescentForLineBundles}, the sheaf $\Des^*\cJ_{g,1}$  is then canonically isomorphic to the invertible sheaf over $\sA_{g,\delta}(\Theta,\delta)$ given by the functor
$$
(\pi:A\rightarrow S,\cL,\beta) \longmapsto (\pi_*\cL)^{\sigma  H}.
$$
By Thm. \ref{theorem:SVNWithMaximalSplitting} (extended to the case where the base is an algebraic stack), we have a canonical isomorphism 
$$
\cV(\delta)\otimes \Des^*\cJ_{g,1} \stackrel{\simeq}\longrightarrow  \cJ_{g,\delta},
$$
as $G(\delta)$-representations, where $\Des^*\cJ_{g,1}$ is given the trivial action. By Thm. 5.1 of \cite{C1}, there is also an isomorphism $\omega_{\Theta}^{-1/2} \simeq \cJ_{g,1}$ of invertible sheaves over $\sX_{g,1}$, canonically defined up to $\pm 1$. The result follows after pulling back this isomorphism by $\Des$.  
\end{proof}

\begin{remark}
In \cite{Mumford:EqAbv2}, \S 6 (bottom of p. 81), Mumford shows that there is an isomorphism $\cJ_{g,\delta}\simeq \cV(\delta)\otimes \cK$ as locally free sheaves over $\sA_{g,\delta}(\Theta,\delta)$, for some invertible sheaf $\cK$. Thm. \ref{theorem:IdealTheorem} above can thus be viewed as a refinement of this result of Mumford, where we have identified $\cK$ explicitly and made the isomorphism almost canonical.
\end{remark}

\subsection{} Let now $\sA_{g,\delta,1/2}(\Theta,\delta)$ be the metaplectic stack over $\sA_{g,\delta}(\Theta,\delta)$. As in Cor. \ref{corollary:metaplecticCharacter}, over this stack the bundle of half-forms factors 
$$
\Des^*\un{\omega}^{1/2}_{\Theta} \simeq \Des^*\cM_{\Theta}^{1/2}\otimes\un{\omega}^{1/2}
$$ 
into the product of the square-root of the theta multiplier bundle and the square-root of the Hodge bundle. Consequently, the isomorphism of Thm. \ref{theorem:IdealTheorem} can be rewritten over the metaplectic stack as
\begin{equation}
\label{equation:algebraicEichlerZagier}
\cV(\delta)\otimes \Des^*\cM_{\Theta}^{-1/2} \stackrel{\simeq}\longrightarrow  \cJ_{\delta,g}\otimes\un{\omega}^{1/2}.
\end{equation}
Now the right-hand side of this equation is a base-change from the metaplectic stack $\sA_{g,\delta,1/2}$, thus the vector bundle $\cV(\delta)\otimes \Des^*\cM_{\Theta}^{-1/2}$ also descends to $\sA_{g,\delta,1/2}$, even though neither one of its factors $\cV(\delta)$ or $\Des^*\cM_{\Theta}^{-1/2}$ does, in general.

\begin{definition}
The {\em Weil bundle} $\cW_{g,\delta}$ is the locally free sheaf of rank $d$ over the metaplectic stack $\sA_{g,\delta,1/2}$ given by
$$
\cW_{g,\delta}:= \cV(\delta)^{\vee}\otimes \Des^*\cM_{\Theta}^{1/2}.
$$ 
\end{definition}

\begin{remark}
Re-writing the isomorphism of Thm. \ref{theorem:IdealTheorem} as \eqref{equation:algebraicEichlerZagier} is helpful in studying the structure of $\cJ_{g,\delta}$. For example, note that $\cW_{g,\delta}^{\vee}$ trivializes over a finite \'{e}tale cover of $\sA_{g,\delta,1/2}$ (i.e. it is a local system over the \'{e}tale site) so that we may deduce that $\cJ_{g,\delta}$ is a {\em projective} local system, i.e. it trivializes over a finite \'{e}tale cover up to multiplication by an invertible sheaf, in this case $\un{\omega}^{1/2}$.  
\end{remark}

The theory of Weil bundles can also be employed to give  an algebraic analog of the notion of {\em vector-valued} modular forms with values in the Weil representation (\cite{EichlerZagier}, \cite{Borcherds}):

\begin{definition}
\label{definition:vvaluedModularForms}
Let $k\in \ZZ$. A {\em $\cW_{g,\delta}$-valued modular form} of weight $k/2$ is a global section of the vector bundle 
$$
\cW_{g,\delta,k/2}:= \cW_{g,\delta}\otimes\un{\omega}^{k/2}
$$
over the metaplectic stack $\sA_{g,\delta,1/2}$.
\end{definition}

The prototypical example of a $\cW_{g,\delta}$-valued modular form is the {\em vector of theta constants}. For any pair $(A\rightarrow S,\cL)$ as in Sec. \ref{section:idealTheoremSection}, this is defined as follows (\cite{FaltingsChai}, Rmk. 5.2). Let $\mathrm{ev}_e\in \Hom(\cL, e_*e^*\cL)$ be the canonical `evaluation-at-$e$' map, and apply $\pi_*$ to get a map in $\Hom(\pi_*\cL,e^*\cL)$. Composed with the normalizing isomorphism $e^*\cL \simeq \cO_S$, we get a canonical element 
$$
\theta_{\rm{null}}(A,\cL) \in \Hom(\pi_*\cL,\cO_S) = H^0(S, (\pi_*\cL)^{\vee}),
$$
the vector of theta constants. Over the moduli stack $\sA_{g,\delta}$, this construction defines an element
$$
\theta_{\mathrm{null},g,\delta} \in  \Hom(\cJ_{g,\delta}, \cO_{\sA_{g,\delta}}) = H^0(\sA_{g,\delta},\cJ_{g,\delta}^{\vee}),
$$
which, via \eqref{equation:algebraicEichlerZagier}, is a $\cW_{g,\delta}$-valued modular form of weight 1/2. This examples is the reason why we choose the convention of the  `dual' superscript in Def. \ref{definition:vvaluedModularForms}.

\subsection{Analytic Theory} We now work in the analytic category and explain how the isomorphism of Thm. \ref{theorem:IdealTheorem} can be viewed as the algebraic analog of the transformation laws of the vector of theta constants. We work out the $g=1$ case, the higher degree cases being similar. Thus, as usual, let $\sA^{\rm{an}}_{1,1} = [\SL_2(\ZZ)\backslash \mathfrak{h}]$ be the analytic quotient stack classifying elliptic curves over an analytic space, let $\pi^{\rm{un}}:\sE(1)\rightarrow\sA^{\rm{an}}_{1,1}$ be the universal elliptic curve over it.  For any $m\in 2\ZZ_{>0}$, we have $\sA^{\rm{an}}_{1,(m)} = \sA^{\rm{an}}_{1,1}$ and the vector bundle $\cJ_{1,(m)}$ is just $\pi^{\rm{un}}_*\cL_m$, where 
$$
\mathcal{L}_m := \mathcal{O}_{\sE(1)}(m\,e)\otimes(\Omega^1_{\sE(1)/\sA^{\rm{an}}_{1,1}})^{\otimes m}.
$$
Let $\mathrm{pr}:\mathfrak{h}\rightarrow \sA^{\rm{an}}_{1,1} = [\SL_2(\ZZ)\backslash \mathfrak{h}]$ be the projection map, classifying the universal framed elliptic curve $\sE\rightarrow\mathfrak{h}$. The vector bundle $\pr^*\cJ_{1,(m)}^{\vee}$ admits a trivialization
$$
\theta_{\mathrm{null},(m)} = (\vartheta_{m,0}, \ldots, \vartheta_{m,m-1}) : \pr^*\cJ_{1,(m)}^{\vee} \stackrel{\simeq}\longrightarrow \cO^{\oplus m}_{\mathfrak{h}},
$$ 
where the vector $\theta_{\mathrm{null},(m)}$ has for components the {\em theta constants}
$$
\vartheta_{m,\nu}(\tau) = \sum_{\substack{r\in\mathbb{Z}\\ r\equiv \nu \mod m}}e^{2\pi i\,\tau\,\frac{r^2}{2m} },
$$
which are non-vanishing on $\mathfrak{h}$. On the other hand, the Weil bundle $\cW_{1,(m)}$, defined over the metaplectic stack $\sA^{\rm{an}}_{1,1/2} = [\Mp_2(\ZZ)\backslash \mathfrak{h}]$, admits a well-known trivialization $\pr^*\cW_{1,(m)}$, which gives the following formulas (e.g. \cite{Borcherds}, \S 4)
\begin{theorem}
\label{theorem:explicitWeilRep}
Let $
T = \left(\smalltwobytwo{1}{1}{0}{1},1\right), S = \left(\smalltwobytwo{0}{-1}{1}{0},\sqrt{\tau}\right),
$
be the standard generators for $\mathrm{Mp}_2(\mathbb{Z})$. Then $\pr^*\cW_{1,(m)}$ admits a trivialization such that the corresponding 1-cocycle in $Z^1(\Mp_2(\ZZ), \GL_m(\cO_{\mathfrak{h}}^{\times}))$ is the Weil representation  $\rho_m:\Mp_2(\ZZ) \rightarrow \GL(\CC[\ZZ/m\ZZ])$ given by  
\begin{equation*}
\begin{aligned}
\rho_m(T)(e_{\gamma}) &= e^{\pi i \gamma^2 /m}\, e_{\gamma} \\
\rho_m(S)(e_{\gamma}) &= \frac{\sqrt{i}}{\sqrt{m}}\sum_{\delta \in \mathbb{Z}/m\mathbb{Z}} e^{-2\pi i \gamma\delta /m}\,e_{\delta},
\end{aligned}
\end{equation*}
where $\{e_{\gamma}\}_{\gamma\in \ZZ/m\ZZ}$ is the standard basis for the group ring $\CC[\ZZ/m\ZZ]$.
\end{theorem}

The composition of the trivialization $\theta_{\mathrm{null},(m)}^{-1}$ with that of Thm. \ref{theorem:explicitWeilRep} (tensored with \eqref{equation:half-trivialization}) gives an isomorphism 
$$
\iota_{\rm{an}}: \pr^*(\cW_{1,(m)}\otimes\un{\omega}^{1/2})\stackrel{\simeq}\longrightarrow \pr^*\cJ^{\vee}_{1,(m)}.
$$ 
By the classical transformation laws of theta functions this isomorphism is $\SL_2(\ZZ)$-invariant, so it descends to $\sA^{\rm{an}}_{1,1}$. On the other hand, we also have an isomorphism $\iota_{\mathrm{alg},\CC}$ between the same sheaves, obtained by dualizing Thm. \ref{theorem:IdealTheorem} and extending scalars to $\CC$. We then have: 

\begin{theorem}
$
\iota_{\rm{an}} = \pm \iota_{\rm{alg},\CC}
$
\end{theorem} 

\begin{proof}
As in the proof of Thm. \ref{theorem:IdealTheorem},  we have that the $G(\cL_m)$-invariant isomorphisms between $\cW_{1,(m)}\otimes\un{\omega}^{1/2}$ and $\cJ^{\vee}_{1,(m)}$ are in bijection with the isomorphisms between $\un{\omega}_{\Theta}^{1/2}$ and  $\cJ^{\vee}_{1,1}$ (this is how $\iota_{\rm{alg}}$ is defined in Thm. \ref{theorem:IdealTheorem}). A quick computation with theta functions then shows that the isomorphism $\iota_{\rm{an}}$ is the one induced by the functional equation of the Riemann-Jacobi theta function $\vartheta(\tau) = \sum_{n\in\ZZ} e^{\pi i n^2\tau}$, as in \cite{C1}, \S 6.1, so the result follows from the analogous result in \em{loc. cit.} 
\end{proof}

\subsection{Jacobi forms and determinants} In this final section we briefly discuss two further applications of Thm. \ref{theorem:IdealTheorem}. The first interesting application is to the theory of Jacobi forms. As shown in \cite{Kramer}, \S 2, the class of $\cL_m$ in $\Pic(\sE(1))$ can be represented by the cocycle
$$
[\smalltwobytwo{a}{b}{c}{d}, \lambda = \lambda_1\tau + \lambda_2]\mapsto e^{2\pi i m \left(\lambda_1^2\tau + 2\lambda_1 z - \frac{c(z + \lambda_1\tau + \lambda_2)^2}{c\tau + d}\right)}.
$$
The global sections of $\cL_m\otimes\left(\Omega^1_{\sE(1)/\sA^{\rm{an}}_{1,1}}\right)^{\otimes k}$ over $\sE(1)$ can thus be identified with {\em Jacobi forms} of index $m/2$ and weight $k\in\ZZ$. Taking global sections of the isomorphism of \eqref{equation:algebraicEichlerZagier} and tensoring with $\CC$ then specializes to the following well-known theorem of Eichler-Zagier (\cite{EichlerZagier}, Thm. 5.1): 

\begin{corollary}
There is a canonical isomorphism between index $m/2$, weight $k$ Jacobi forms and $\cW^{\vee}_{1,(m)}$-valued modular forms of weight $k-1/2$ (Def. \ref{definition:vvaluedModularForms}).  
\end{corollary}

Thm. \ref{theorem:IdealTheorem}, and more precisely \eqref{equation:algebraicEichlerZagier}, can thus be viewed as a generalization of the Eichler-Zagier isomorphism in two directions: first, it generalizes to Jacobi forms of any type (index) and degree; second, it is purely algebraic. Thm. \ref{theorem:IdealTheorem} also has applications to the theory of determinants of abelian schemes. Indeed, taking determinants of both sides of \eqref{equation:algebraicEichlerZagier}, we get
$$
\det \cW^{\vee}_{g,\delta} \simeq \det \cJ_{g,\delta}\otimes\un{\omega}^{-d/2}.
$$ 
Note that since $d$ is even, both determinants are actually defined over the stack $\sA_{g,\delta}$, and not just over the metaplectic stack. Moreover we recognize the square of the right-hand side as the {\em determinant bundle} (\cite{MoretBailly:PinceauxAbv}, VIII.1.0.5, \cite{FaltingsChai}, \S 5)
$$
\det \cJ_{g,\delta}^{\otimes 2}\otimes\un{\omega}^{-d} =: \Delta_{g,\delta}
$$
over the moduli stack $\sA_{g,\delta}$. In particular, we have
\begin{corollary}
The order of $\Delta_{g,\delta}$ in $\Pic(\sA_{g,\delta})$ is the same as the order of $\left(\det \cW^{\vee}_{g,\delta}\right)^{\otimes 2}$. 
\end{corollary}

The question of finding the precise order of $\Delta_{g,\delta}$ has been studied in \cite{FaltingsChai}, Thm. 5.1, \cite{Polishchuck:Determinants}, \cite{Kouvidakis}  and \cite{Maillot-Rossler}. In particular, our approach to the question can be viewed as the algebraic version of \cite{Kouvidakis}, where the transformation laws of theta functions are employed to compute the order of $\Delta_{g,\delta}$.

\renewcommand\refname{References}

\bibliographystyle{plain}
\bibliography{thetaPaper}

\end{document}